\documentclass[11pt]{article}
\usepackage{amsmath}
\usepackage{amssymb}
\usepackage{amsbsy}
\usepackage{amsthm}
\usepackage{epsfig}
\usepackage{wrapfig}
\usepackage{color,cancel}
\usepackage{fullpage}
\usepackage{float}
\usepackage{ulem}
\usepackage{breqn}
\usepackage{graphicx}
\usepackage{caption}
\usepackage{subcaption}
\usepackage{bbm}
\usepackage{bigints}

\usepackage{verbatim}

\usepackage{url} 

\usepackage{amsmath,amsfonts,amssymb,color,graphicx}

\newtheorem{theorem}{Theorem}[section]

\newtheorem{remark}{Remark}[section]

\numberwithin{equation}{section}
\newcommand{\norm}[1]{\left\Vert#1\right\Vert}

\newcommand{\bu}{{\bf u}}
\newcommand{\bv}{{\bf v}}

\newcommand{\bw}{{\bf w}}
\newcommand{\bx}{{\bf x}}

\newcommand{\bff}{{\bf f}}
\newcommand{\bfg}{{\bf g}}

\newcommand{\bfX}{{\bf X}}
\newcommand{\bfV}{{\bf V}}

\usepackage{color}
\title{Stokes with variable viscosity}
\usepackage[section]{placeins}
\makeatletter
\AtBeginDocument{%
	\expandafter\renewcommand\expandafter\subsection\expandafter{%
		\expandafter\@fb@secFB\subsection
	}%
}

\date{}

\title{Long time stability of a linearly extrapolated blended BDF scheme for multiphysics flows }

\author{
	Aytekin \c{C}{\i}b{\i}k
	\thanks{Department of Mathematics, Gazi University, 06550 Ankara, Turkey; abayram@gazi.edu.tr}
	\and
Fatma G. Eroglu
\thanks{Department of Mathematics, Middle East Technical University, 06800 Ankara, Turkey, Department of Mathematics, Faculty of Science,  Bart{\i}n University, 74110 Bart{\i}n, Turkey; fguler@bartin.edu.tr}
\and
Songul Kaya
\thanks{
Department of Mathematics, Middle East Technical University, 06800 Ankara, Turkey; smerdan@metu.edu.tr.}
}

\begin{document}

\maketitle

\begin{abstract}
This paper investigates the long time stability behavior of multiphysics flow problems, namely the Navier-Stokes equations, natural convection and double-diffusive convection equations with an extrapolated blended BDF time-stepping scheme. This scheme combines the two-step BDF and three-step BDF time stepping schemes. We prove unconditional long time stability theorems for each of flow systems. Various numerical tests are given for large time step sizes in long time intervals in order to support theoretical results.
\end{abstract}



\section{Introduction}

Most of the engineering and applied science problem involves the combination of some different physical problems such as fluid flow, heat transfer, mass transfer, magnetic and electricity effect. These kinds of problems are mostly referred as multiphysics problems. From a mathematical point of view, these problems yield systems of coupled single physics equations. In our case, many important applications require the accurate solution of multiphysics coupling with Navier-Stokes equations. Since the simulation of Navier-Stokes equations has its own difficulties, the coupling between involved equations yield more complex problems. One possibility of improving numerical simulations is to develop algorithms which are reliable and robust. In addition, designed numerical schemes should capture the long time dynamics of the system in a right way. Thus, it is of practical interest to have an algorithm which is stable over the required long time intervals.

In recent years, considerable amount of effort has been spent to understand the long time behavior of the numerical  schemes for multiphysics problems. For such works, we refer to \cite{mr6,mr1,mr2,mr5,mr3,mr4}. In particular, for Navier-Stokes equations, the Crank-Nicholson  in \cite{hr86,HR4,mr2}, the implicit Euler  in \cite{mr3}, two-step Backward Differentiation (BDF2) in \cite{minelong}  and  fractional step/projection methods in \cite{sa94} are chosen as temporal discretization  in order to show the long time stability. In this respect, an extrapolated two step BDF scheme for a velocity-vorticity form of Navier-Stokes equations have been investigated in \cite{heist}. The long time stability of partitioned methods for the fully evolutionary Stokes-Darcy problem in \cite{l13}, for the time dependent MHD system in \cite{l14,mr5} and for the double-diffusive convection in \cite{F15} were also established based on implicit-explicit and backward differentiation schemes.

This study concerns the behavior of the solutions of multiphysics problems for longer time simulations and time step restrictions on these solutions. In order to solve the systems numerically, a finite element in space discretization is employed along with a rather new time stepping approach so called an extrapolated blended Backward Differentiation (BLEBDF) temporal discretization. Such scheme combines BDF2 and a three-step BDF method in order to not only preserve $A$-stability and second order accuracy but also have a smaller constant in truncation error terms, \cite{vatsa}. Along with the mentioned time-stepping strategy, the three-step extrapolation is used in order to linearize the convective terms in the system of PDE's. Thus, the solution of only one linear system of equations is encountered at each time step which reduces the compilation time and memory cost in simulations.

In \cite{R15}, Ravindran considers the stability and convergence of a double diffusive convection system with the blended BDF scheme in short time intervals. The current study attempts to extend the works above to study the notion of long time stability by combining the BLEBDF idea and its effects on several multiphysics flows such as Navier-Stokes, natural convection and double-diffusive convection. In this work, we will provide the unconditonal long time $L^2$ stability property of BLEBDF method for each of flow systems, when they are discretized spatially with finite element method. To the best of authors' knowledge, this is the first study that investigates the long time stability of the finite element solutions of multiphysics flows involving a blended BDF time-stepping approach along with a third order linear extrapolation idea.

The plan of the paper is as follows. In Section 2, we state the notations with some mathematical preliminaries. Section 3 is reserved for proving the unconditional long time stability of NSE under the employment of extrapolated blended BDF temporal discretization along with numerical experiments. Similar results are established for natural convection and double-diffusive convection equations in Section 4 and Section 5, respectively. Finally, we state some conclusion remarks in the last section.

\section{Notations and Preliminaries}

Let $\Omega \subset \mathbb{R}^d,\,d\in\{2,3\}$, be open, connected domain bounded by Lipschitz boundary $\partial \Omega$. Throughout the paper standard notations for Sobolev spaces and their norms will be used, c.f. Adams \cite{A75}.  The norm in $(H^k(\Omega))^d$ is denoted by $\|\cdot\|_k$ and the norm in Lebesgue spaces $(L^p(\Omega))^d$, $1\leq p < \infty$, $p\neq 2$ by $\|\cdot\|_{L^p}$ and $p=\infty$ by $\|\cdot\|_{\infty}$ . The space $L^2(\Omega)$ is equipped with the norm and inner product $\|\cdot\|$ and $(\cdot, \cdot)$, respectively, and for these we drop the subscripts. Vector-valued functions will be identified by bold face. The norm in dual space  ${H}^{-1}$ of ${H}^1_0(\Omega)$ is denoted by $\norm{\cdot}_{-1}$. The continuous velocity, pressure, temperature and concentration spaces are denoted by
\begin{eqnarray*}
	\bfX&:=& ({\bf H}^1_0(\Omega))^d, \ Q:=L_0^2(\Omega), \ W:={H}^1_0(\Omega), \ \Psi:={H}^1_0(\Omega),
\end{eqnarray*}  and the divergence free space
\[
\bfV= \lbrace \bv \in \bfX: (\nabla \cdot \bv, q)=0, \forall q\in Q\rbrace.
\]
We recall also the Poincar\'e-Friedrichs inequality as
$$\|\bv\|\leq C_P \|\nabla v\|,\quad \forall \bv \in \bfX.$$
For each multiphysics flow problems, we consider a regular, conforming family $\Pi^h$ of triangulations of domain with maximum diameter $h$  for spatial discretization. Assume  ${\bfX_h}\subset {\bfX}, Q_h\subset Q, W_h\subset W$ and $\Psi_h\subset \Psi$ be finite element spaces such that the spaces $(\bfX_h,Q_h)$ satisfy the discrete inf-sup condition needed for stability of the discrete pressure, \cite{GR79}.  The discretely divergence free space for $ (\bfX_h,Q_h)$ pairs is given by
\begin{equation}
\bfV_h= \lbrace \bv_h \in \bfX_h: (\nabla \cdot \bv_h, q_h)=0, \forall q_h \in Q_h\rbrace.
\end{equation}
The dual spaces of $\bfV_h$, $W_h$ and $\Psi_h$ are given by $\bfV_h^*$, $W_h^*$ and $\Psi_h^*$, and their norms are denoted by $\norm{\cdot}_{ \bfV_h^*}$, $\norm{\cdot}_{W_h^*}$ and  $\norm{\cdot}_{\Psi_h^*}$   respectively.
We also  need  the following space in the analysis
\begin{eqnarray}
L^{\infty}(\mathbb{R}_+,\bfV_h^*):=\{\bff:\Omega^d\times \mathbb{R}_+\rightarrow\mathbb{R}^d, \exists \, K <\infty,\, a.e. \,\, t>0,\,\, \|\bff(t)\|_{\bfV_h^*}<K\}.
\end{eqnarray}
Similar spaces for $W_h^*$ and $\Psi_h^*$ will be used throughout the analysis. To simplify the analysis, we utilize the $G$-stability framework as in \cite{HW02}. For third order backward differentiation, the positive definite matrix $G$-matrix and the associated norm  can be obtained as
\[ G=\dfrac{1}{12}\left( \begin{array}{ccc}
19& -12 &3 \\
-12& 10& -3\\
3 & -3 & 1\end{array} \right), \quad  \|\mathbf{\mathcal{W}}\|^2_{G}=(\mathcal{W},G\mathcal{W}). \quad \mathcal{W}\in (L^2(\Omega))\]
It is easy to see the $G$-norm and $L^2$ norm are equivalent in the sense that there exist $0 <C_l<C_u $ positive constants such that
\begin{eqnarray}
\begin{array}{rclll}
C_l\| \mathcal{W}\|_G^2&\leq& \| \mathcal{W}\|^2&\leq& C_u \| \mathcal{W}\|_G^2.\label{eqv}
\end{array}
\end{eqnarray}
The following relation is well known (see, e.g. \cite{R15}) for $\forall \bw^j \in L^2(\Omega)$, ${ \bf \mathcal{W}}_{n+1}=[\bw_h^{n+1} \quad \bw_h^{n} \quad \bw_h^{n-1}]^\top$ and ${ \bf \mathcal{W}}_{n}=[\bw_h^{n} \quad \bw_h^{n-1} \quad \bw_h^{n-2}]^\top$.
	\begin{eqnarray}
\left(\frac{5}{3}\bw_h^{n+1}-\frac{5}{2}\bw_h^n+\bw_h^{n-1}-\frac{1}{6}\bw_h^{n-2},\bw_h^{n+1}\right)
&=&\|{ \bf \mathcal{W}}_{n+1}\|_G^2-\| { \bf \mathcal{W}}_{n}\|_G^2\nonumber\\&&+\dfrac{1}{12}\|\bw_h^{n+1}-3\bw_h^{n}+3\bw_h^{n-1}-\bw_h^{n-2}\|^2.\label{ey}
\end{eqnarray} 	

\section{Long time stability of NSE with BLEBDF }

The goal of this section is to show that our scheme is unconditionally long-time stable. That is, the solutions remain bounded without any time step restriction. We first study an extrapolated blended BDF method for discretizing the incompressible NSE: 
\begin{equation}\label{nse1}
\begin {array}{rcll}
\bu_t -\nu \Delta \bu+ (\bu\cdot\nabla)\bu + \nabla p &=& \bff  &
\mathrm{in }\quad  \Omega, \\
\nabla \cdot \bu&=& 0& \mathrm{in } \quad \Omega,\\
\bu&=& \mathbf{0}& \mathrm{on } \quad \partial \Omega,\\
\bu(0,\bx) & = & \bu_0 & \mathrm{in }\quad \Omega, \\
\displaystyle \int_{\Omega}p \ d \bx&=&0.
\end{array}
\end{equation}
where $\bu$ is the velocity field, $p$ is the fluid pressure, $\nu$ is the kinematics viscosity and
$\bff$ denotes the body forces. The variational formulation of (\ref{nse1}) reads as follows:
Find $(\bu,p) \in (\bfX,Q)$ satisfying
\begin{equation}\label{weaknse}
\begin {array}{r@{}l}
(\bu_t,\bv) +\nu(\nabla{\bu},\nabla{\bv}) + b_1(\bu,\bu,\bv) -(p,\nabla \cdot \bv)&{}=(\bff,\bv), \\
{(\nabla \cdot \bu,q)}&{}=0,
\end{array}
\end{equation}
for all $ (\bv,q) \in (X,Q)$, where
\begin{eqnarray}
b_1(\bu,\bv,\bw) &:=&
\frac{1}{2}\left(((\bu\cdot\nabla)\bv,\bw)-((\bu\cdot\nabla)\bw,\bv)\right)
\end{eqnarray}
represents the skew-symmetric form of the convective term. Note that the convective term has the well known property
\begin{equation}
b_1(\bu,\bv,\bv)=0 \label{s1}
\end{equation}
for all $\bu,\bv\in \bfX$, which simplifies the analysis.

The studied time discretization method uses  different time discretizations for different terms. The scheme discretizes in time via a BDF2 and BDF3 whereas the nonlinear terms are treated via a third-order extrapolation formula. One consequence of lagging the nonlinear term to previous time levels is to avoid solving nonlinear equations. A class of this type of blending BDF schemes which are also known as a class of optimized second-order BDF schemes are proposed in \cite{N03,vatsa}. We now consider one of the special case of this family proposed in \cite{vatsa}, with an error constant half as large as BDF2 scheme.

Based on the weak formulation (\ref{nse1}), the fully discrete approximation of it reads as follows. Given $\bf f$,
$\bu_h^0=\bu_h^{-1}=\bu_h^{-2}=\bu_0$, for any time step $\Delta t>0$,  find  $(\bu_h^{n+1},p_h^{n+1})\in (\bfX_h,Q_h)$ such that for $n\geq 0$
\begin{eqnarray}
\left(\dfrac{\frac{5}{3}\bu_h^{n+1}-\frac{5}{2}\bu_h^n+\bu_h^{n-1}-\frac{1}{6}\bu_h^{n-2}}{\Delta t},\bv_h\right) + \nu(\nabla{\bu_h^{n+1}},\nabla{\bv_h})&+&b_1(3\bu_h^n-3\bu_h^{n-1}+\bu_h^{n-2},\bu_h^{n+1},\bv_h)
 \nonumber\\
-(p_h^{n+1},\nabla \cdot \bv_h)&=&(\bff^{n+1},\bv_h), \label{b1}
\\
(\nabla\cdot\bu_h^{n+1},q_h)&=&0 \label{b2}
\end{eqnarray}
for all $(\bv_h,q_h)\in (\bfX_h,Q_h)$.

We now analyze the long time stability over $0\leq t^n  < \infty$ and show that the scheme is uniformly bounded in time
in $L^2$ norm. Our proof is based on a $G$-norm and the associated  estimation (\ref{ey}).

\begin{theorem}(Unconditional long time  stability of NSE in $L^2$) \label{Lem:sta1} Assume $\bff \in L^{\infty}(\mathbb{R}_+,\bfV_h^*) $, then the approximation (\ref{b1})-(\ref{b2}) is long time stable in the following sense: for any $\Delta t>0$,
	\begin{eqnarray}
	\lefteqn{\|{ \bf \mathcal{U}}_{n+1}\|_G^2	 +\dfrac{\nu\Delta t}{4}\|\nabla{\bu_h^{n+1}}\|^2+\dfrac{\nu\Delta t}{16}\|\nabla{\bu_h^{n}}\|^2\nonumber\leq(1+\alpha)^{-(n+1)}(\|{ \bf \mathcal{U}}_{0}\|_G^2}\\
	&&+\dfrac{\nu\Delta t}{4}\|\nabla{\bu_h^{0}}\|^2+\dfrac{\nu\Delta t}{16}\|\nabla{\bu_h^{-1}}\|^2)+\max\{\dfrac{8C_p^2}{C_l\nu^2}, \dfrac{2\nu^{-1}\Delta t}{3}\}
\|\bff\|^2_{L^{\infty}(\mathbb{R}_+;\bfV_h^*)}.\label{lemma}
	\end{eqnarray}
	where $\mathcal{U}_0=[\bu_h^0 \quad \bu_h^{-1} \quad \bu_h^{-2}]$ and $\alpha=min\{\dfrac{C_l\nu \Delta t}{16C_p^2}, \dfrac{3}{4}\}$.
\end{theorem}
\begin{remark}
Theorem \label{Lem:sta1}implies that the result holds for any large $n$, since
the constants are independent of $n$.
\end{remark}

\begin{proof}
Setting $\bv_h=\bu_h^{n+1}$ in (\ref{b1}) and $q_h=p_h^{n+1}$ in (\ref{b2}) the  BLEBDF scheme and using the skew-symmetry property (\ref{s1}), we obtain
\begin{eqnarray}
\left(\dfrac{\frac{5}{3}\bu_h^{n+1}-\frac{5}{2}\bu_h^n+\bu_h^{n-1}-\frac{1}{6}\bu_h^{n-2}}{\Delta t},\bu_h^{n+1}\right) + \nu(\nabla{\bu_h^{n+1}},\nabla\bu_h^{n+1})
	=(\bff^{n+1},\bu_h^{n+1}).\label{bl}
	\end{eqnarray}
From (\ref{ey}), we get	
\begin{eqnarray}
\|{ \bf \mathcal{U}}_{n+1}\|_G^2-\|{ \bf \mathcal{U}}_{n}\|_G^2+\dfrac{1}{12}\|\bu_h^{n+1}-3\bu_h^{n}+3\bu_h^{n-1}-\bu_h^{n-2}\|^2	 +\nu\Delta t\|\nabla{\bu_h^{n+1}}\|^2
=\Delta t (\bff^{n+1},\bu_h^{n+1}).\label{bll}
\end{eqnarray}
where ${ \bf \mathcal{U}}_{n+1}=[\bu_h^{n+1} \quad \bu_h^{n} \quad \bu_h^{n-1}]^\top$ and ${ \bf \mathcal{U}}_{n}=[\bu_h^{n} \quad \bu_h^{n-1} \quad \bu_h^{n-2}]^\top$.
Applying  the Cauchy- Schwarz and Young's inequality for the right hand side of the (\ref{bll}) gives
\begin{eqnarray}
\|{ \bf \mathcal{U}}_{n+1}\|_G^2-\|{ \bf \mathcal{U}}_{n}\|_G^2+\dfrac{1}{12}\|\bu_h^{n+1}-3\bu_h^{n}+3\bu_h^{n-1}-\bu_h^{n-2}\|^2	 +\dfrac{\nu\Delta t}{2}\|\nabla{\bu_h^{n+1}}\|^2
\leq\dfrac{\nu^{-1}\Delta t}{2} \|\bff^{n+1}\|_{\bfV_h^*}^2 \label{s3}
	\end{eqnarray}
Adding both of side $\dfrac{\nu\Delta t}{4}\|\nabla{\bu_h^{n}}\|^2+\dfrac{\nu\Delta t}{16}\|\nabla{\bu_h^{n-1}}\|^2$ and discarding the third positive term in the left hand side of (\ref{s3}) yields
\begin{eqnarray}
\left(\|{ \bf \mathcal{U}}_{n+1}\|_G^2	 +\dfrac{\nu\Delta t}{4}\|\nabla{\bu_h^{n+1}}\|^2+\dfrac{\nu\Delta t}{16}\|\nabla{\bu_h^{n}}\|^2\right )+	\dfrac{\nu\Delta t}{16}\left(\|\nabla{\bu_h^{n+1}}\|^2+\|\nabla{\bu_h^{n}}\|^2+\|\nabla{\bu_h^{n-1}}\|^2\right)
\nonumber\\
+\dfrac{3\nu\Delta t}{16}\|\nabla{\bu_h^{n+1}}\|^2+\dfrac{\nu\Delta t}{8}\|\nabla{\bu_h^{n}}\|^2 \nonumber
\\
\leq\|{ \bf \mathcal{U}}_{n}\|_G^2+\dfrac{\nu\Delta t}{4}\|\nabla{\bu_h^{n}}\|^2
+\dfrac{\nu\Delta t}{16}\|\nabla{\bu_h^{n-1}}\|^2
+\dfrac{\nu^{-1}\Delta t}{2} \|\bff^{n+1}\|_{\bfV_h^*}^2.\label{s15}
	\end{eqnarray}
The terms in the left hand side of (\ref{s15}) can be rearranged by applying the Poincar\'e-Friedrichs and the equivalent norm property (\ref{eqv}):  
\begin{eqnarray}
\lefteqn{	\dfrac{\nu\Delta t}{16}\left(\|\nabla{\bu_h^{n+1}}\|^2+\|\nabla{\bu_h^{n}}\|^2+\|\nabla{\bu_h^{n-1}}\|^2\right)+\dfrac{3\nu\Delta t}{16}\|\nabla{\bu_h^{n+1}}\|^2+\dfrac{\nu\Delta t}{8}\|\nabla{\bu_h^{n}}\|^2}\nonumber\\
	&\geq& \dfrac{C_l\nu \Delta t}{16C_p^2}\|{ \bf \mathcal{U}}_{n+1}\|_G^2+\dfrac{3\nu\Delta t}{16}\|\nabla{\bu_h^{n+1}}\|^2+\dfrac{\nu\Delta t}{8}\|\nabla{\bu_h^{n}}\|^2\nonumber\\
	&\geq& \alpha (\|{ \bf \mathcal{U}}_{n+1}\|_G^2+\dfrac{\nu\Delta t}{4}\|\nabla{\bu_h^{n+1}}\|^2+\dfrac{\nu\Delta t}{16}\|\nabla{\bu_h^{n}}\|^2),\label{blb}
	\end{eqnarray}
	where $\alpha=min\{\dfrac{C_l\nu \Delta t}{16C_p^2},\dfrac{3}{4} \}$.
Inserting (\ref{blb}) in (\ref{s15}) and using induction along with  $\|\bff^{i}\|_{\bfV_h^*}^2\leq \|\bff\|^2_{L^{\infty}(\mathbb{R}_+;\bfV_h^*)}$, $\forall i=1, \dots, n+1$ gives
\begin{eqnarray}
\lefteqn{(\|{ \bf \mathcal{U}}_{n+1}\|_G^2+\dfrac{\nu\Delta t}{4}\|\nabla{\bu_h^{n+1}}\|^2+\dfrac{\nu\Delta t}{16}\|\nabla{\bu_h^{n}}\|^2)}\nonumber\\&\leq&(1+\alpha)^{-(n+1)}(\|{ \bf \mathcal{U}}_{0}\|_G^2+\dfrac{\nu\Delta t}{4}\|\nabla{\bu_h^{0}}\|^2+\dfrac{\nu\Delta t}{16}\|\nabla{\bu_h^{-1}}\|^2)\nonumber\\&&+\dfrac{(1+\alpha)^{-1}\nu^{-1}\Delta t}{2}
\lefteqn{	\|\bff\|^2_{L^{\infty}(\mathbb{R}_+;\bfV_h^*)}	\Big((1+\alpha)^{-n}+(1+\alpha)^{-(n-1)} +\dots+1\Big) } \nonumber
\end{eqnarray}
Since $|\dfrac{1}{1+\alpha}|< 1$, then we get

\begin{eqnarray}
\lefteqn{(\|{ \bf \mathcal{U}}_{n+1}\|_G^2+\dfrac{\nu\Delta t}{4}\|\nabla{\bu_h^{n+1}}\|^2+\dfrac{\nu\Delta t}{16}\|\nabla{\bu_h^{n}}\|^2)}\nonumber
\\
&\leq&(1+\alpha)^{-(n+1)}(\|{ \bf \mathcal{U}}_{0}\|_G^2+\dfrac{\nu\Delta t}{4}\|\nabla{\bu_h^{0}}\|^2+\dfrac{\nu\Delta t}{16}\|\nabla{\bu_h^{-1}}\|^2)\nonumber\\
&&+ \max \{ \dfrac{8C_p^2}{C_l\nu^2}, \dfrac{2\nu^{-1}\Delta t}{3}\} \|\bff\|^2_{L^{\infty}(\mathbb{R}_+;\bfV_h^*)}\label{bl11}
\end{eqnarray}
which is the required result.

\end{proof}

\subsection{Numerical experiment for NSE}

Throughout this paper, we carry out tests for each flow problem separately.  We expose the evolution of the $L^{2}$ norm of solution variables in long time intervals for each variable according to related problems. After each numerical test, we provide a table showing the CPU-times of relevant simulations according to varying $\Delta t$ and compare these CPU-times with classical BDF2 methods' CPU-times in order to reveal the computational advantage of the method. We choose the inf-sup stable Scott-Vogelious finite element pair for velocity-pressure couple, which is known to be discretely divergence free. We refer \cite{sv}, for the details of this selection.  Throughout all our simulations, we use public license finite element software package FreeFem++, \cite{hec}.

We now perform a numerical test in order to verify theoretical long time stability result obtained in Theorem \ref{Lem:sta1}. We set $\Omega=(0,1)^2$ with a coarse mesh resolution of $16 \times 16$.
We calculate the approximate solution in time interval $[0,400]$ and calculate the $L^{2}$ norms for different $\Delta t$ and varying $\nu$. We pick the initial velocity and the forcing term for this test problem as :
\begin{align}\label{truesol}
\textbf{u}=\left(%
\begin{array}{c}
\sin(\pi x)\sin(\pi y)\\
\cos(\pi x)\cos(\pi y)
\end{array}%
\right),\qquad
f= \left(\begin{array}{c}
 y^{2}cos(xy^{2})+\sin(x)\sin(y)\\
 2xycos(xy^{2})+\cos(x)\cos(y)
\end{array}  \right).
\end{align}
In Figure \ref{fig:nse}, we present the evolution of $\|\textbf{u}^{n+1}\|$ in long time interval $[0,400]$. We give the stability results for two large time step instances for each case. As could be easily deduced from the figure, the solutions we obtain from the proposed scheme are long time stable even by using large time step sizes. Although slight deviations seem to occur due to small viscosity for the case $\nu = 0.001$, still the solution is within an interval of stability. Notice that, these solutions are obtained for a very coarse mesh resolution of $16 \times 16$. Hence, small oscillations inside a narrow interval does not degrade the stability of the solution.
\begin{figure}[h!]
	\centerline{\hbox{
				\includegraphics[width=0.45\linewidth]{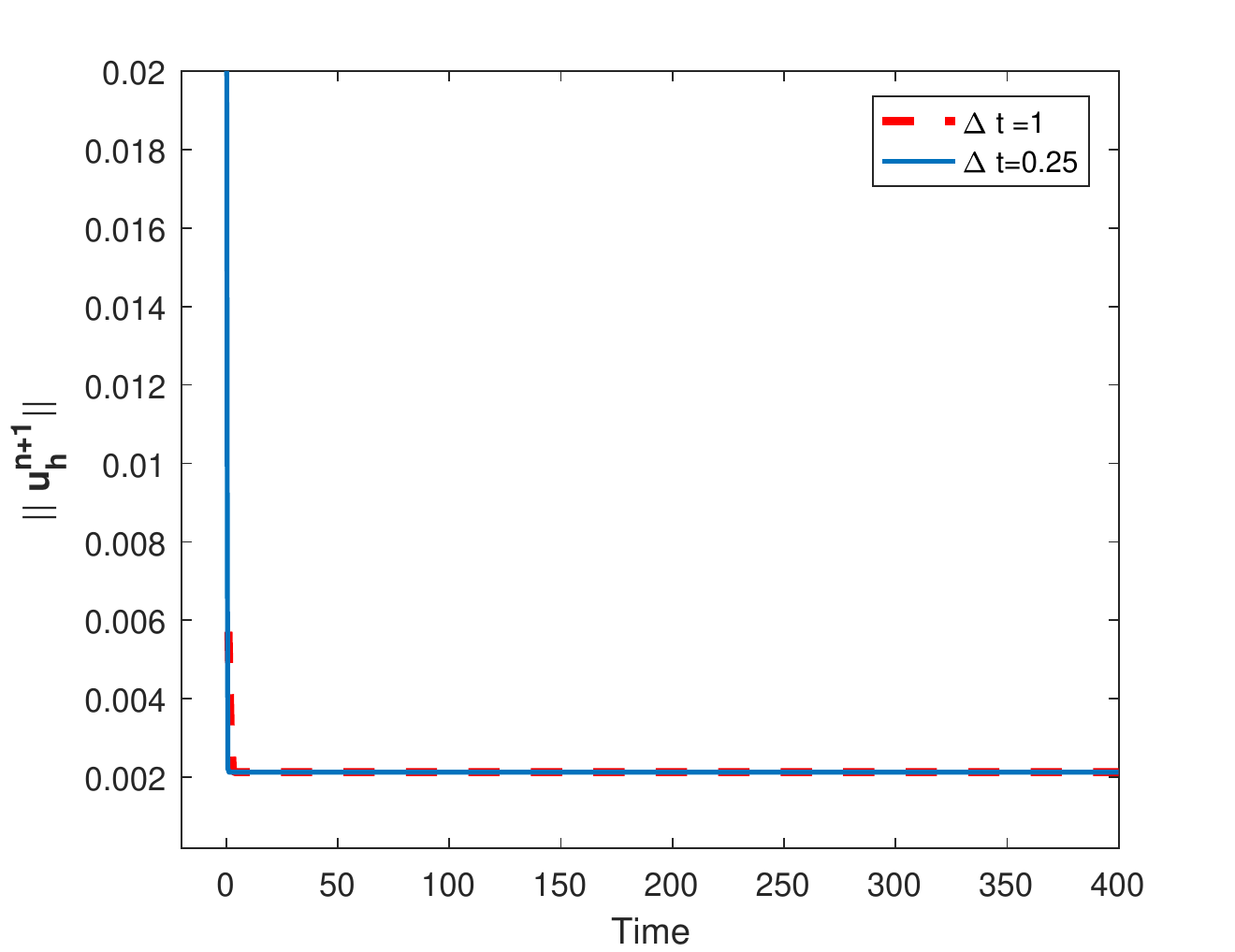}
			\includegraphics[width=0.45\linewidth]{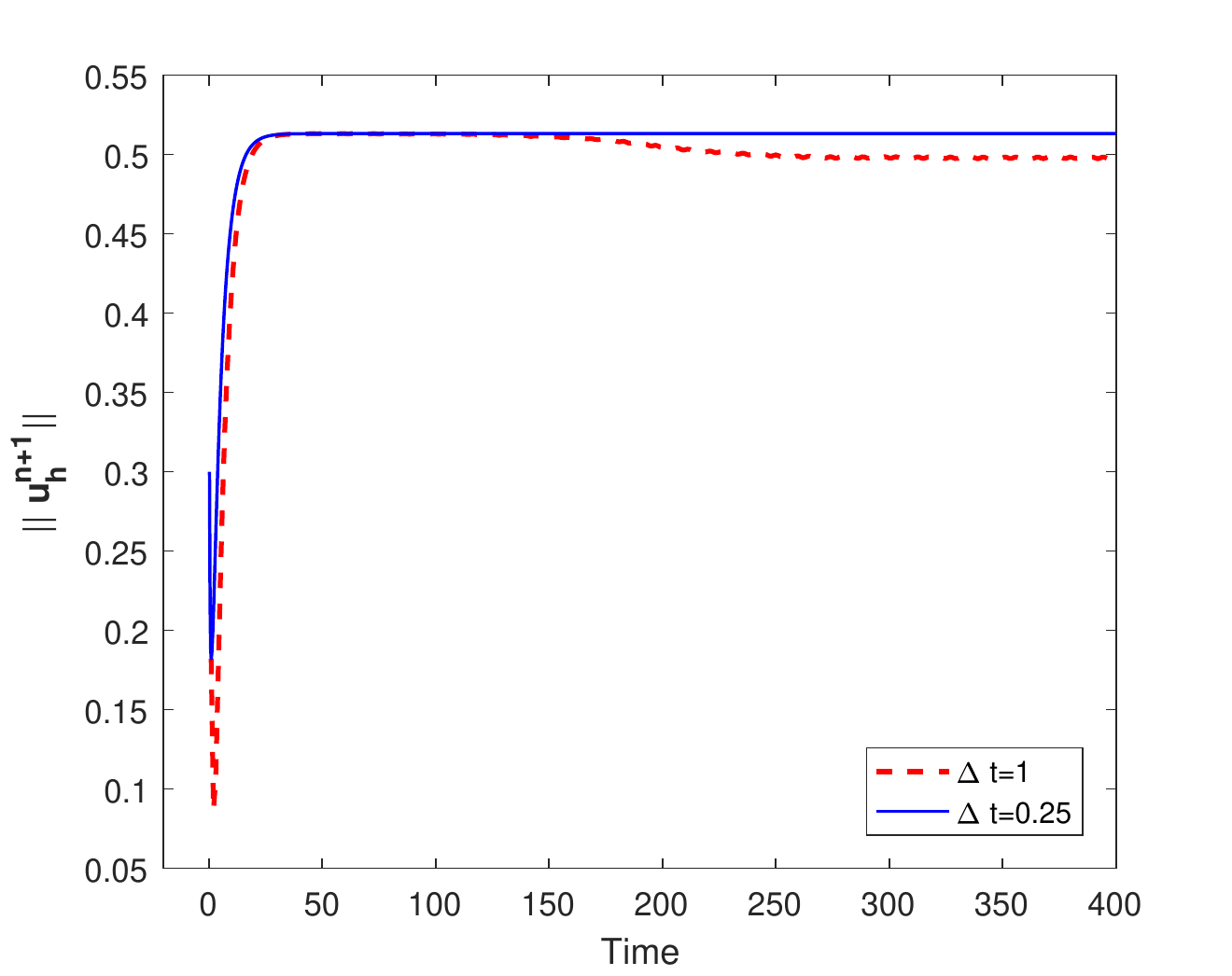}	
	}}
	\centerline{\hbox{
			\includegraphics[width=0.45\linewidth]{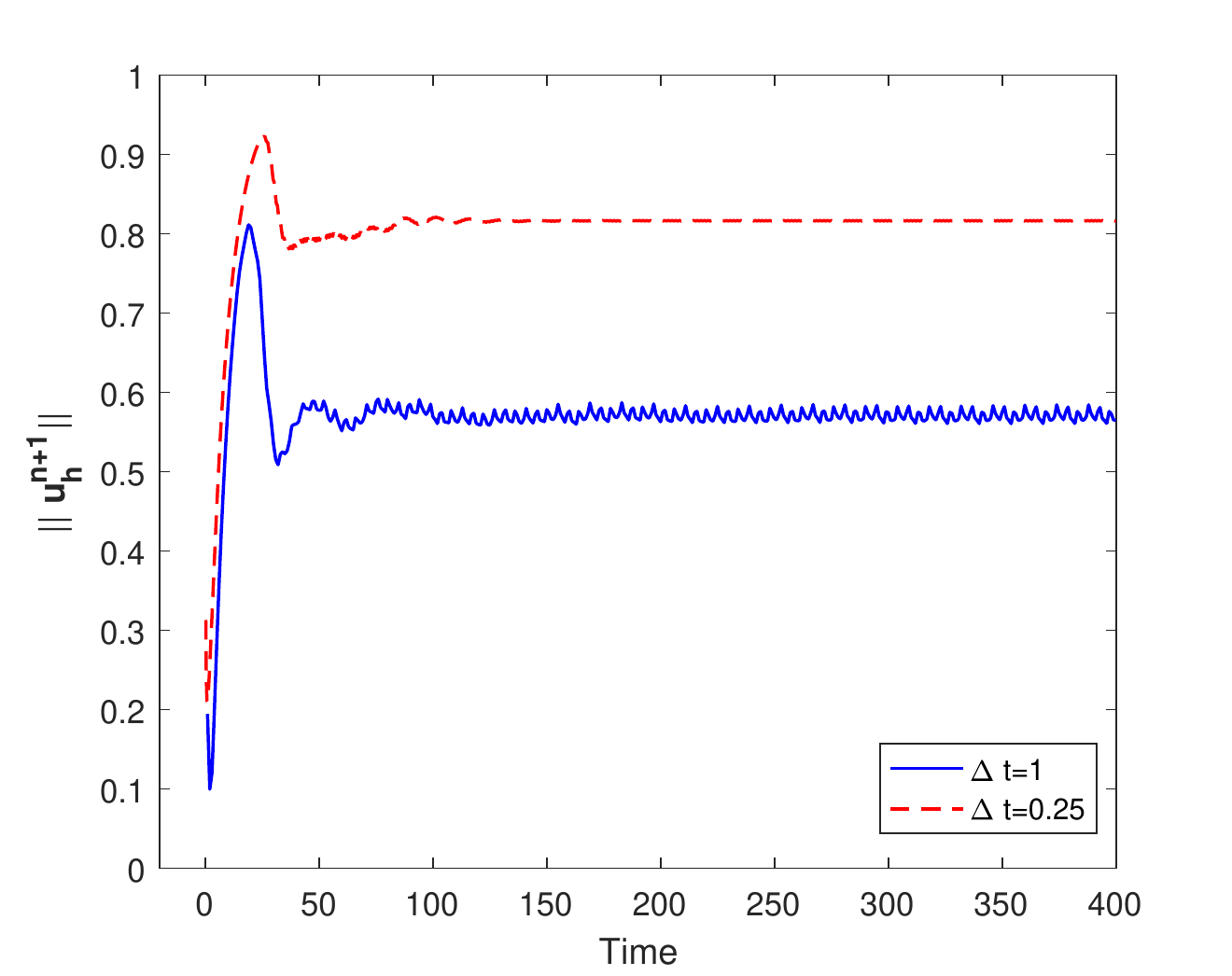}
			\includegraphics[width=0.45\linewidth]{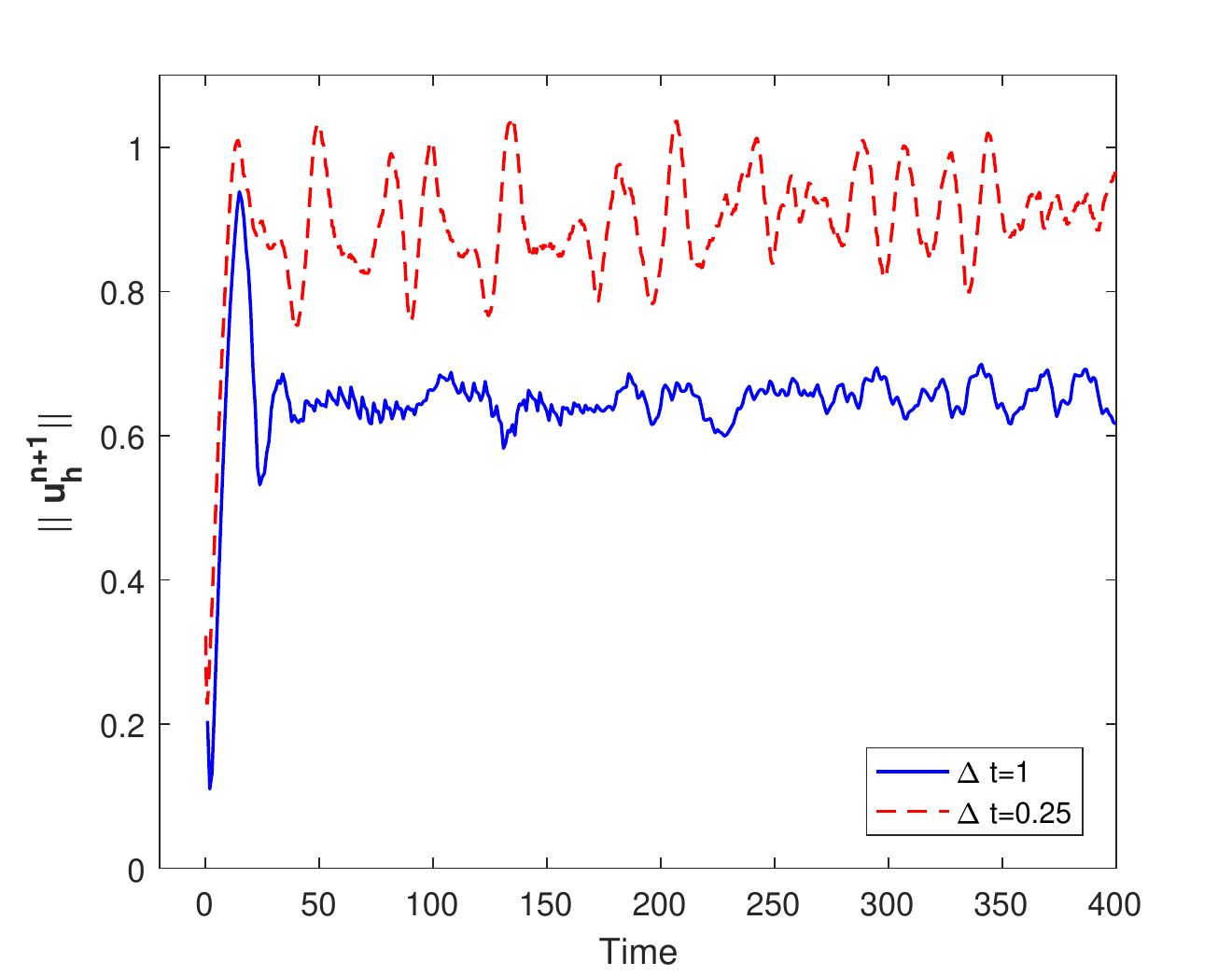}
	}}
	\caption{\label{fig:nse} Evolution of the $L^{2}$ norm of the velocity solution for varying $\nu$. $\nu = 1,\,\,\nu = 0.004$ (upper left to right) and  $\nu = 0.002,\,\,\nu = 0.001$ (lower left to right). }
\end{figure}

We also compare the CPU times of the scheme (\ref{b1})-(\ref{b2}) and run the same problem in a shorter time interval $[0,150]$. As could be observed from Table \ref{table:cpunse}, the BLEBDF scheme has an advantage in terms of simulation CPU times when compared with a classical BDF2 scheme. When $\Delta t$ becomes smaller, the difference between the CPU times are increasing. This shows the promise of the method, especially when smaller time step sizes are used.
\begin{table}[hh!]
\centering
\begin{tabular}{|c c c|}
\hline
  $\Delta t$ & BDF2 & BLEBDF \\ \hline\hline
  1 & 1.37 & 0.87 \\
  0.1 & 12 & 8 \\
  0.01 & 131.5& 78.7\\
  \hline
\end{tabular}
\caption{ Comparison of the CPU-times (seconds) of classical BDF2 scheme and BLEBDF for the NSE with $\nu=0.001$ on a time interval $[0,150]$.}
\label{table:cpunse}
\end{table}

\section{Long time stability of natural convection equations with BLEBDF}

In this section, we prove that the BLEBDF scheme  for natural convection equations
is also long time stable. The unsteady natural convection system in $\Omega$ with partitioned boundary $\partial \Omega=\Gamma_T\cup\Gamma_B$ with $\Gamma_T\cap\Gamma_B = \emptyset $, the so called Boussinesq equations are given by
\begin{equation}\label{mnc}
\begin {array}{rcll}
\bu_{t}-\nu \Delta \bu+ (\bu\cdot\nabla)\bu + \nabla p &=& Ri\,T{\bf {e}}_{2} + f   &
\mathrm{in}\ \Omega, \\
\nabla \cdot \bu&=& 0& \mathrm{in}\  \Omega,\\
T_t - \nabla\cdot(\kappa \nabla T )+ (\bu\cdot\nabla)T  &=& \gamma &
\mathrm{in }\  \Omega, \\
\bu(0,\bx)  =  \bu_0, T(0,\bx) & = & T_0 & \mathrm{in }\ \Omega,
\\
\bu=\mathbf{0}\,\, \mathrm{on}\  \partial \Omega,\,\,
T={0}\ \mathrm{on}\ \Gamma_T, &\dfrac{\partial T}{\partial {\bf n}}&=0& \mathrm{on} \ \Gamma_B,
\end{array}
\end{equation}
where $\bu$,$p$, $T$ are the fluid velocity, the pressure and the temperature, respectively. The parameters in \eqref{mnc} are the kinematic viscosity $\nu$, the thermal conductivity parameter $\kappa >0$ and the Richardson number $Ri$ and the unit vector is given by ${\bf e}_2$. The prescribed body forces are $\bff$ and $\gamma$. The initial velocity and temperature are $\bu_0$ and $T_0$, respectively.

By using similar notations as in Section 3, given $\bff$, $\gamma$, $\bu_h^0=\bu_h^{-1}=\bu_h^{-2}=\bu_0$ and $T_h^0 = T_h^{-1} = T_h^{-2} = T_0$ find $(\bu_{h}^{n+1},p_{h}^{n+1},T_{h}^{n+1}) \in (\bfX_h,Q_h,W_h)$ for $n\geq 1$ satisfying
\begin{gather}
\left(\dfrac{\frac{5}{3}\bu_h^{n+1}-\frac{5}{2}\bu_h^n+\bu_h^{n-1}-\frac{1}{6}\bu_h^{n-2}}{\Delta t},\bv_h\right) + \nu(\nabla{\bu_h^{n+1}},\nabla{\bv_h})+b_1(\bu^*,\bu_h^{n+1},\bv_h) \nonumber
\\
-(p_h,\nabla \cdot \bu_h^{n+1})=Ri\, ( T^*{\bf e}_2,\bv_h)+ (\bff^{n+1},\bv_h), \label{mnub}\\
(q_h,\nabla \cdot \bv_h^{n+1})=0 \label{mnup} \\
{\left(\dfrac{\frac{5}{3}T_h^{n+1}-\frac{5}{2}T_h^n+T_h^{n-1}-\frac{1}{6}T_h^{n-2}}{\Delta t},S_h\right)
+\kappa (\nabla T_h^{n+1}, \nabla S_h)+b_2(\bu^*,T_h^{n+1},S_h)} \nonumber \\
=(\gamma^{n+1}, S_h),\label{mntb}
\end{gather}
for all $(\bv_h,q_h,S_h) \in (\bfX_h,Q_h,W_h)$ where $\bu^*=3\bu_h^n-3\bu_h^{n-1}+\bu_h^{n-2}$ and $T^*=3T_h^n-3T_h^{n-1}+T_h^{n-2}$. Herein the related skew-symmetric form is given by
\begin{eqnarray}
b_2(\bu,T,S) &:=&
\frac{1}{2}\left(((\bu\cdot\nabla)
T,S)-((\bu\cdot\nabla)S,T)\right)
\end{eqnarray}
for all $\bu\in \bfX, T,S\in W.$

\begin{theorem}(Unconditional long time  stability of natural convection in $L^2$) \label{Lem:stamn}
Assume $\bff\in L^{\infty}(\mathbb{R}_+,\bfV_h^*) $ and $\gamma \in L^{\infty}(\mathbb{R}_+,W_h^*) $, then the approximation (\ref{mnub})-(\ref{mntb})
is long time stable  in the following sense: for any $\Delta t>0$,
\begin{eqnarray}
\lefteqn{	\|{ \bf \mathcal{U}}_{n+1}\|_G^2+	\|{ \bf \mathcal{T}}_{n+1}\|_G^2 +\dfrac{\nu\Delta t}{4}\|\nabla{\bu_h^{n+1}}\|^2+\dfrac{\kappa\Delta t}{4}\|\nabla{T_h^{n+1}}\|^2\leq(1+\alpha)^{-(n+1)}\Big(\|{ \bf \mathcal{U}}_{0}\|_G^2+\dfrac{\nu\Delta t}{4}\|\nabla{\bu_h^{0}}\|^2}\nonumber\\
	&&+\dfrac{\nu\Delta t}{16}\|\nabla{\bu_h^{-1}}\|^2\Big)	+  (K_{\alpha}+(1+\beta)^{-1})\Big( (1+\beta)^{-n}(\|{ \bf \mathcal{T}}_{0}\|_G^2+\dfrac{\kappa\Delta t}{4}\|\nabla{T_h^{0}}\|^2+\dfrac{\kappa\Delta t}{16}\|\nabla{T_h^{-1}}\|^2)\Big)
	\nonumber\\
	&&+(K_{\alpha}+1)\max \{ \dfrac{8C_p^2}{C_l\kappa^2}, \dfrac{2\kappa^{-1}\Delta t}{3}\} \|\gamma\|_{L^{\infty}(\mathbb{R}_+,W_h^*)}^2\nonumber\\
	&&+\max\{\dfrac{16C_p^2}{C_l\nu^2}, \dfrac{ 4\nu^{-1}\Delta t}{3}\}\|\bff\|_{L^{\infty}(\mathbb{R}_+,\bfV_h^*)}	 \label{lemmu}
	\end{eqnarray}
	where  $K_{\alpha}=\dfrac{49C_uC_p^2\nu^{-1}Ri^2\Delta t}{\alpha}$, $\mathcal{U}_0=[\bu_h^0 \quad \bu_h^{-1} \quad \bu_h^{-2}]$,  $\mathcal{T}_0=[T_h^0 \quad T_h^{-1} \quad T_h^{-2}]$, $\alpha=min\{\dfrac{C_l\nu \Delta t}{16C_p^2},\dfrac{3}{4}\}$,  $\beta=min\{\dfrac{C_l\kappa \Delta t}{16C_p^2},\dfrac{3}{4} \}$,
\end{theorem}

\begin{proof}
The proof starts with a stability bound for temperature letting $S_h=T_h^{n+1}$ in (\ref{mntb}) and
using the skew symmetry property $b_2(\bu^{*},T_h^{n+1},T_h^{n+1})=0$. Along with (\ref{ey}) and multiplying with $\Delta t$, yields
\begin{eqnarray}
\|{ \bf \mathcal{T}}_{n+1}\|_G^2-\| { \bf \mathcal{T}}_{n}\|_G^2+\dfrac{1}{12}\|T_h^{n+1}-3T_h^{n}+3T_h^{n-1}-T_h^{n-2}\|^2+\kappa \Delta t\|\nabla T_h^{n+1}\|^2=\Delta t(\gamma^{n+1}, T_h^{n+1}).\quad \label{mnt1}
	\end{eqnarray}
where ${ \bf \mathcal{T}}_{n+1}=[T_h^{n+1} \quad T_h^{n} \quad T_h^{n-1}]^\top$ and ${ \bf \mathcal{T}}_{n}=[T_h^{n} \quad T_h^{n-1} \quad T_h^{n-2}]^\top$.
Applying the Cauchy-Schwarz, and Young's inequalities leads to
\begin{eqnarray}
\|{ \bf \mathcal{T}}_{n+1}\|_G^2-\|{ \bf \mathcal{T}}_{n}\|_G^2+\dfrac{1}{12}\|T_h^{n+1}-3T_h^{n}+3T_h^{n-1}-T_h^{n-2}\|^2	 +\dfrac{\kappa\Delta t}{2}\|\nabla{T_h^{n+1}}\|^2\leq \frac{\kappa^{-1}\Delta t}{2}\|\gamma^{n+1}\|_{W_h^*}^2.\label{blt}
	\end{eqnarray}
Adding both of side  $\dfrac{\kappa\Delta t}{4}\|\nabla{T_h^{n}}\|^2+\dfrac{\kappa\Delta t}{16}\|\nabla{T_h^{n-1}}\|^2$ and dropping the nonnegative terms produce
	\begin{gather}
(\|{  \mathcal{T}}_{n+1}\|_G^2	 +\dfrac{\kappa\Delta t}{4}\|\nabla{T_h^{n+1}}\|^2+\dfrac{\kappa\Delta t}{16}\|\nabla{T_h^{n}}\|^2) +  \dfrac{\kappa\Delta t}{16} \left ( \|\nabla{T_h^{n+1}}\|^2+\|\nabla{T_h^{n}}\|^2+ \|\nabla{T_h^{n-1}}\|^2\right) \nonumber
\\
+\dfrac{3\kappa\Delta t}{16}\|\nabla{T_h^{n+1}}\|^2+\dfrac{\kappa\Delta t}{8}\|\nabla{T_h^{n}}\|^2
\leq\|{  \mathcal{T}}_{n}\|_G^2+\dfrac{\kappa\Delta t}{4}\|\nabla{T_h^{n}}\|^2
+\dfrac{\kappa\Delta t}{16}\|\nabla{T_h^{n-1}}\|^2+\frac{\kappa^{-1}\Delta t}{2}\|\gamma^{n+1}\|_{W_h^*}^2.\label{bl2t}
	\end{gather}
The estimation of (\ref{bl2t}) follows closely that of (\ref{blb}). Again, the last five terms of the left hand side of (\ref{bl2t}) can be rearranged by the applying the Poincar\'e-Friedrichs and the equivalent norm property (\ref{eqv}) as
	\begin{eqnarray}
\lefteqn{\dfrac{\kappa\Delta t}{16}( \|\nabla{T_h^{n+1}}\|^2+\|\nabla{T_h^{n}}\|^2+\|\nabla{T_h^{n-1}}\|^2)+\dfrac{3\kappa\Delta t}{16}\|\nabla{T_h^{n+1}}\|^2+\dfrac{\kappa\Delta t}{8}\|\nabla{T_h^{n}}\|^2 }\nonumber
\\
&\geq& \beta (\|{ \bf \mathcal{T}}_{n+1}\|_G^2+\dfrac{\kappa\Delta t}{4}\|\nabla{T_h^{n+1}}\|^2+\dfrac{\kappa\Delta t}{16}\|\nabla{T_h^{n}}\|^2),\label{blbt}
	\end{eqnarray}
where $\beta=min\{\dfrac{C_l\kappa \Delta t}{16C_p^2}, \dfrac{3}{4}\}$.
Arguing as before and inserting (\ref{blbt}) in (\ref{bl2t}) and using induction leads to
	\begin{eqnarray}
\lefteqn{	\|{ \bf \mathcal{T}}_{n+1}\|_G^2	 +\dfrac{\kappa\Delta t}{4}\|\nabla{T_h^{n+1}}\|^2+\dfrac{\kappa\Delta t}{16}\|\nabla{T_h^{n}}\|^2}\nonumber
	\\
	&\leq&(1+\beta)^{-(n+1)}(\|{ \bf \mathcal{T}}_{0}\|_G^2+\dfrac{\kappa\Delta t}{4}\|\nabla{T_h^{0}}\|^2
	+\dfrac{\kappa\Delta t}{16}\|\nabla{T_h^{-1}}\|^2)+\dfrac{\kappa^{-1}\Delta t}{2\beta} \|\gamma\|_{L^{\infty}(\mathbb{R}_+,W_h^*)}^2,\label{bl5tn}
	\end{eqnarray}
which is the result for long the time stability for the temperature.
Letting $\bv_h=\bu_h^{n+1}$ in (\ref{mnub}), and $q_h=p_h^{n+1}$ in
(\ref{mnup}), one obtains
\begin{eqnarray}
\left(\dfrac{\frac{5}{3}\bu_h^{n+1}-\frac{5}{2}\bu_h^n+\bu_h^{n-1}-\frac{1}{6}\bu_h^{n-2}}{\Delta t},\bu_h^{n+1}\right) + \nu\|\nabla{\bu_h^{n+1}}\|^2=Ri\, (T^{*}{\bf e}_2,\bu_h^{n+1})+ (\bff^{n+1},\bu_h^{n+1}).\label{blu}
\end{eqnarray}
Repeating the arguments used obtaining (\ref{s3}) yields
	\begin{eqnarray}
\lefteqn{	\|{ \bf \mathcal{U}}_{n+1}\|_G^2-\|{ \bf \mathcal{U}}_{n}\|_G^2+\dfrac{1}{12}\|\bu_h^{n+1}-3\bu_h^{n}+3\bu_h^{n-1}-\bu_h^{n-2}\|^2	 +\dfrac{\nu\Delta t}{2}\|\nabla{\bu_h^{n+1}}\|^2}\nonumber\\
	&\leq&C_p^2\nu^{-1}Ri^2 \Delta t \|T^*\|^2\|{\bf e}_2\|^2+\nu^{-1} \Delta t\|\bff^{n+1}\|_{\bfV_h^*}
	\label{blu1}
	\end{eqnarray}
Note that ${\bf e}_2$ is a unit vector, i.e., $\|{\bf e}_2\|=1$ and using the definition of $\mathcal{T}^*$ and (\ref{eqv}), we get $\|\mathcal{T}^*\|\leq 7\|\mathcal{T}_n\|\leq 7\sqrt{C_u} \|\mathcal{T}_n\|_G$, we get
		\begin{gather}
	\|{ \bf \mathcal{U}}_{n+1}\|_G^2-\|{ \bf \mathcal{U}}_{n}\|_G^2+\dfrac{1}{12}\|\bu_h^{n+1}-3\bu_h^{n}+3\bu_h^{n-1}-\bu_h^{n-2}\|^2	 +\dfrac{\nu\Delta t}{2}\|\nabla{\bu_h^{n+1}\|^2}\nonumber\\
	\leq 49C_uC_p^2\nu^{-1}Ri^2  \Delta t \|\mathcal{T}_{n}\|_G^2+\nu^{-1} \Delta t\|\bff^{n+1}\|_{\bfV_h^*}
	\label{bl3}
	\end{gather}
Arguing as in (\ref{s15}), one gets the following estimation for (\ref{bl3})
	\begin{eqnarray}
\lefteqn{(1+\alpha)(\|{ \bf \mathcal{U}}_{n+1}\|_G^2	 +\dfrac{\nu\Delta t}{4}\|\nabla{\bu_h^{n+1}}\|^2+\dfrac{\nu\Delta t}{16}\|\nabla{\bu_h^{n}}\|^2)}\nonumber\\
&\leq&(\|{ \bf \mathcal{U}}_{n}\|_G^2+\dfrac{\nu\Delta t}{4}\|\nabla{\bu_h^{n}}\|^2+\dfrac{\nu\Delta t}{16}\|\nabla{\bu_h^{n-1}}\|^2)\nonumber
\\
&&+49C_uC_p^2\nu^{-1}Ri^2 \Delta t \|\mathcal{T}_{n}\|_G^2 +\nu^{-1} \Delta t\|\bff^{n+1}\|_{\bfV_h^*} \label{blu2}
	\end{eqnarray}
where $\alpha=min\{\dfrac{C_l\nu \Delta t}{16C_p^2},\dfrac{3}{4}\}$.	Putting $n$ instead of $(n+1)$ in (\ref{bl5tn}) and inserting it in (\ref{blu2}) yields
	\begin{eqnarray}
	\lefteqn{(1+\alpha)(\|{ \bf \mathcal{U}}_{n+1}\|_G^2	 +\dfrac{\nu\Delta t}{4}\|\nabla{\bu_h^{n+1}}\|^2+\dfrac{\nu\Delta t}{16}\|\nabla{\bu_h^{n}}\|^2)} \nonumber	
	\\
	&\leq&(\|{ \bf \mathcal{U}}_{n}\|_G^2+\dfrac{\nu\Delta t}{4}\|\nabla{\bu_h^{n}\|^2}+\dfrac{\nu\Delta t}{16}\|\nabla{\bu_h^{n-1}}\|^2)\nonumber
	\\
&&	+49C_uC_p^2\nu^{-1}Ri^2 \Delta t \Big( (1+\beta)^{-n}(\|{ \bf \mathcal{T}}_{0}\|_G^2+\dfrac{\kappa\Delta t}{4}\|\nabla{T_h^{0}}\|^2\nonumber\\&&+\dfrac{\kappa\Delta t}{16}\|\nabla{T_h^{-1}}\|^2)+\dfrac{\kappa^{-1}\Delta t}{2\beta} \|\gamma\|_{L^{\infty}(\mathbb{R}_+,W_h^*)}^2
	\Big) +\nu^{-1} \Delta t\|\bff^{n+1}\|_{\bfV_h^*}\label{blu3}
	\end{eqnarray}
Applying induction adding (\ref{bl5tn}) to (\ref{blu3}) completes the proof theorem.

\end{proof}

\subsection{Numerical experiment for natural convection equations}

In this subsection, we present the long time numerical simulation results of a coupled
Navier-Stokes system given by the scheme (\ref{mnub})-(\ref{mntb}).  We consider natural convection equations in $\Omega=(0,1)^2$ with a coarse mesh resolution of $16 \times 16$. For this problem, we solve a buoyancy driven cavity test in which, upper and lower boundaries of the square domain are kept adiabatic and vertical boundaries are kept at different temperatures imposed as Dirichlet boundary conditions. The flow initiates naturally by density differences due to the temperature variations at opposing vertical boundaries along with the gravitational force. The domain presented in Figure \ref{fig:ncdomain} is used in the numerical test. Velocity boundary conditions are no-slip everywhere. We set the Richardson number, $Ri=1$ and perform the computation in the time interval $[0,150]$. The results for different viscosities and time step sizes are provided. We examine the time evolution of the computed solutions in discrete norms $\|\bu_h^n\|$ and $\|T_h^n\|$.
\begin{figure}[htb]
\centerline{\hbox{
\includegraphics[width=0.45\linewidth]{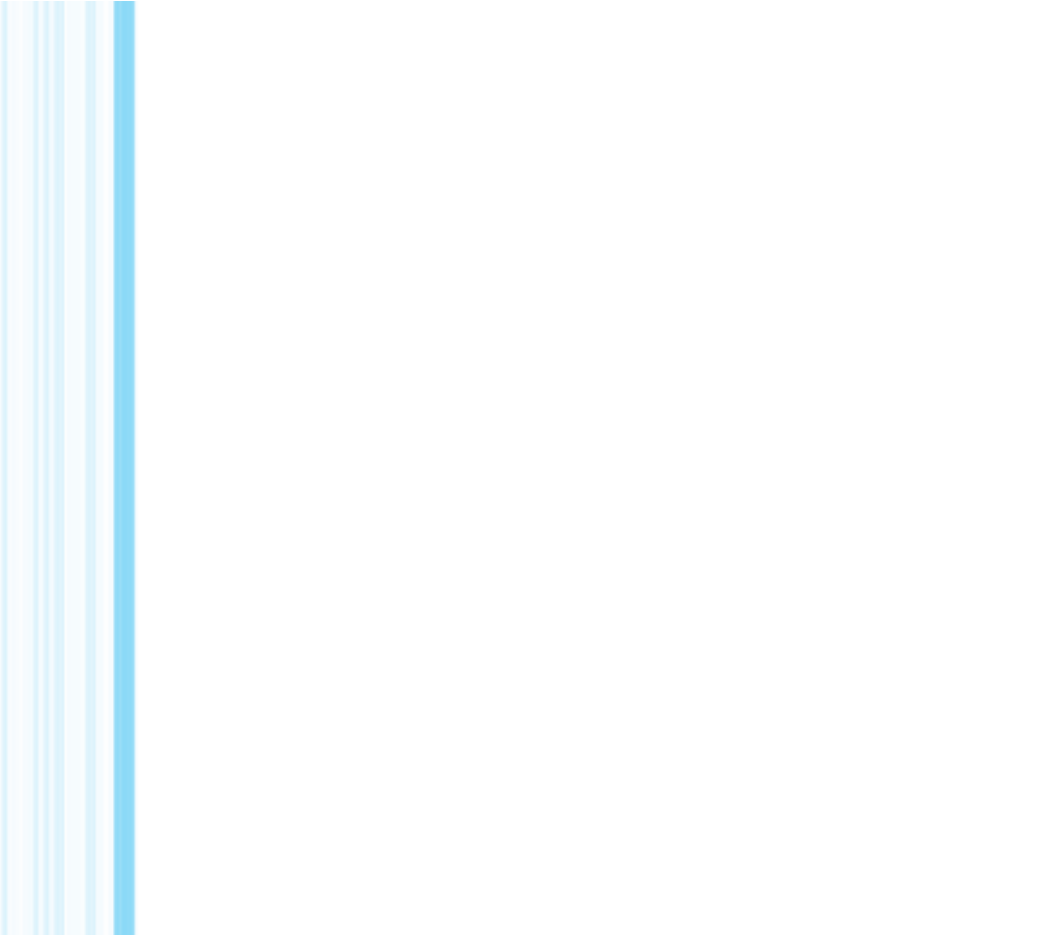}
}}
\caption{\label{fig:ncdomain} The computational domain for the natural convection test example.}
\end{figure}
\begin{figure}[h!]
	\centerline{\hbox{
				\includegraphics[width=0.45\linewidth]{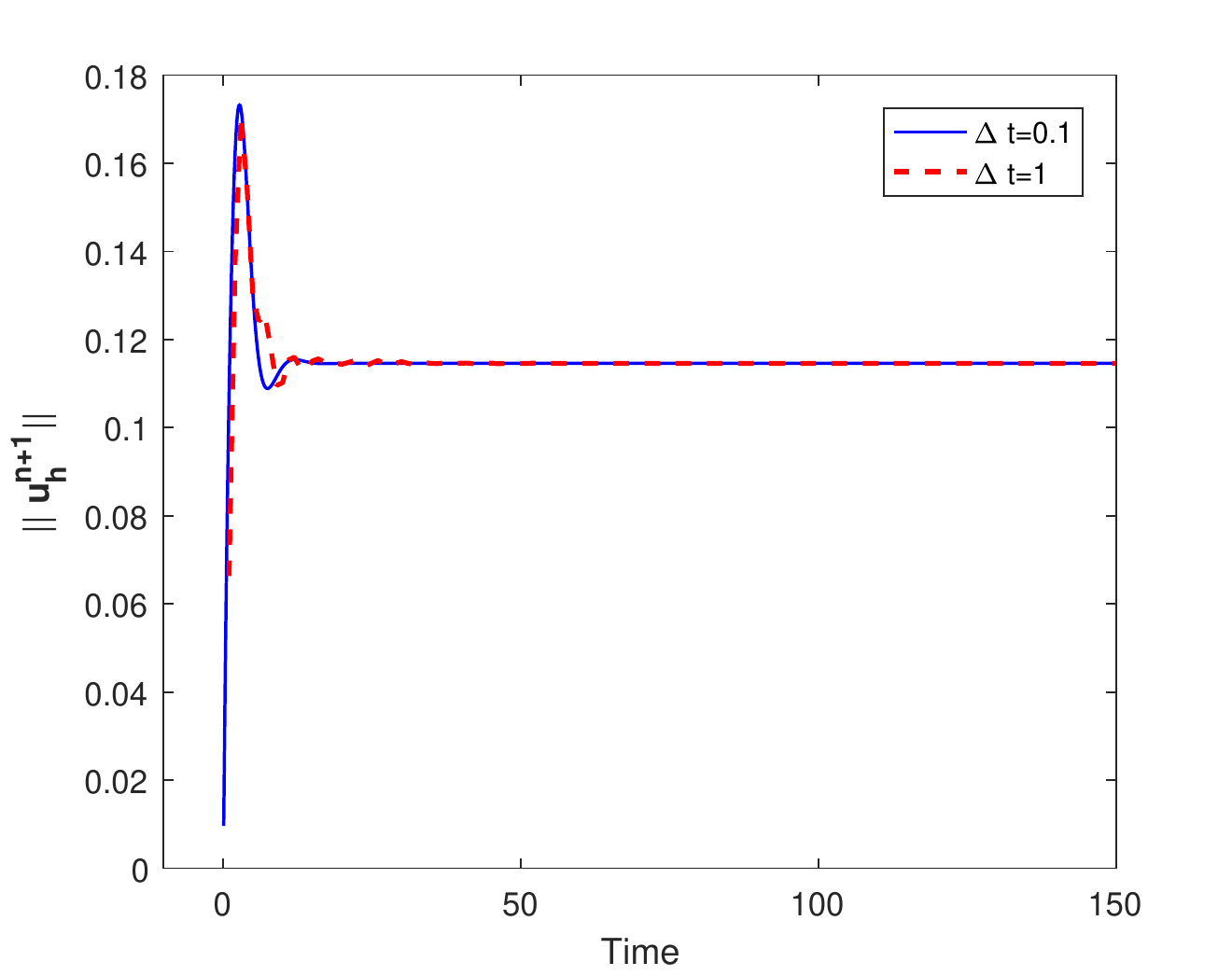}
				\includegraphics[width=0.45\linewidth]{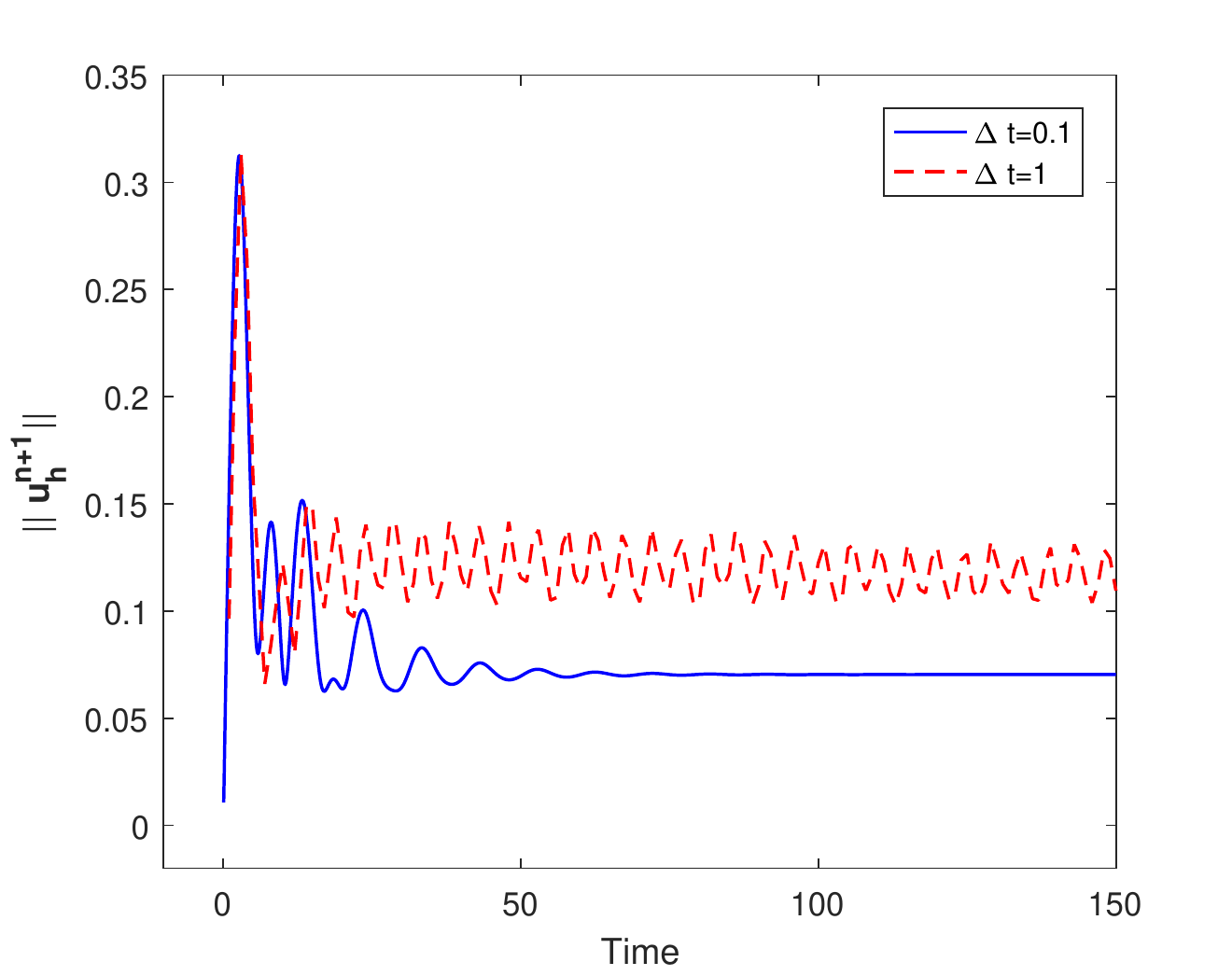}	
	}}
	\centerline{\hbox{
			\includegraphics[width=0.45\linewidth]{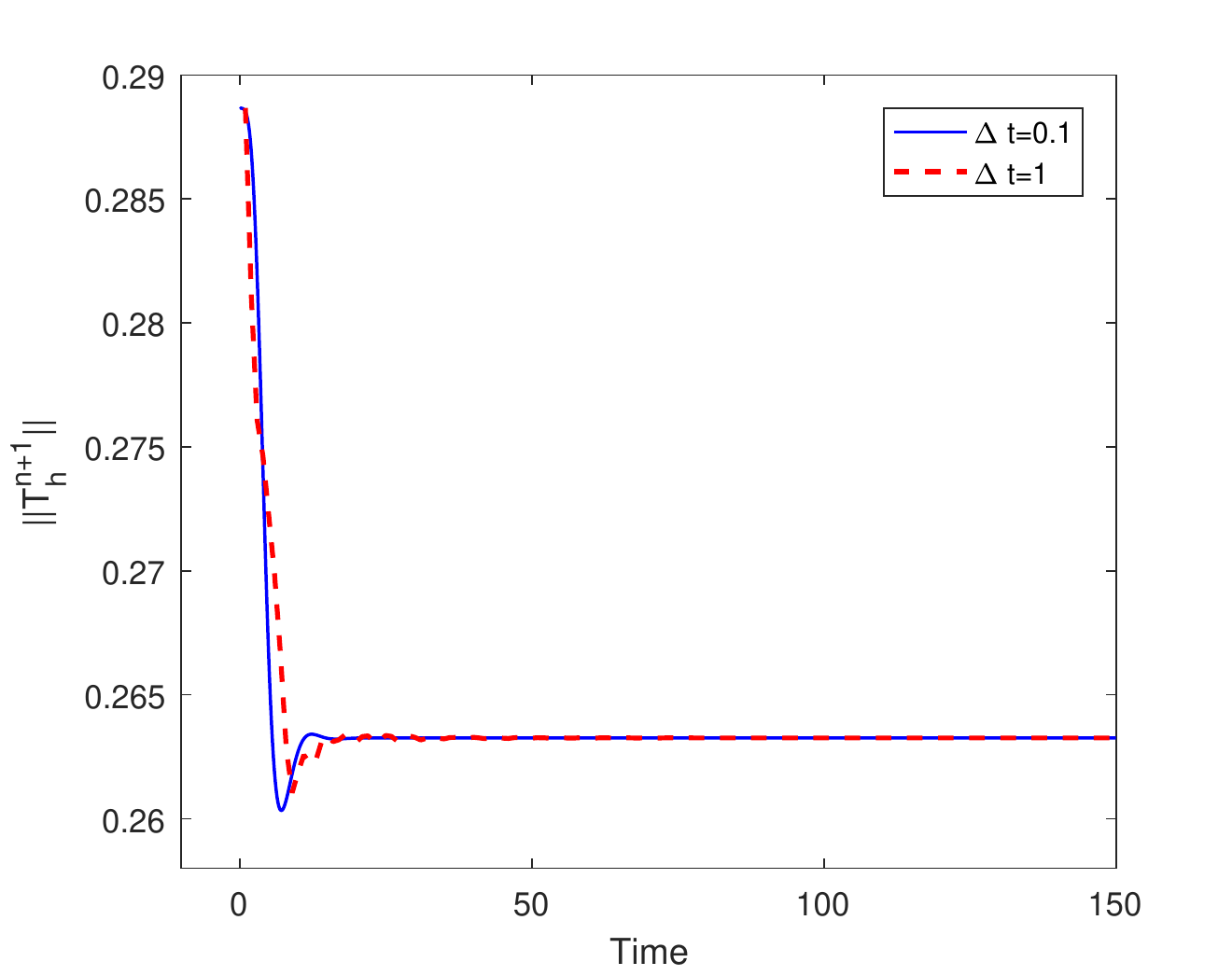}
			\includegraphics[width=0.45\linewidth]{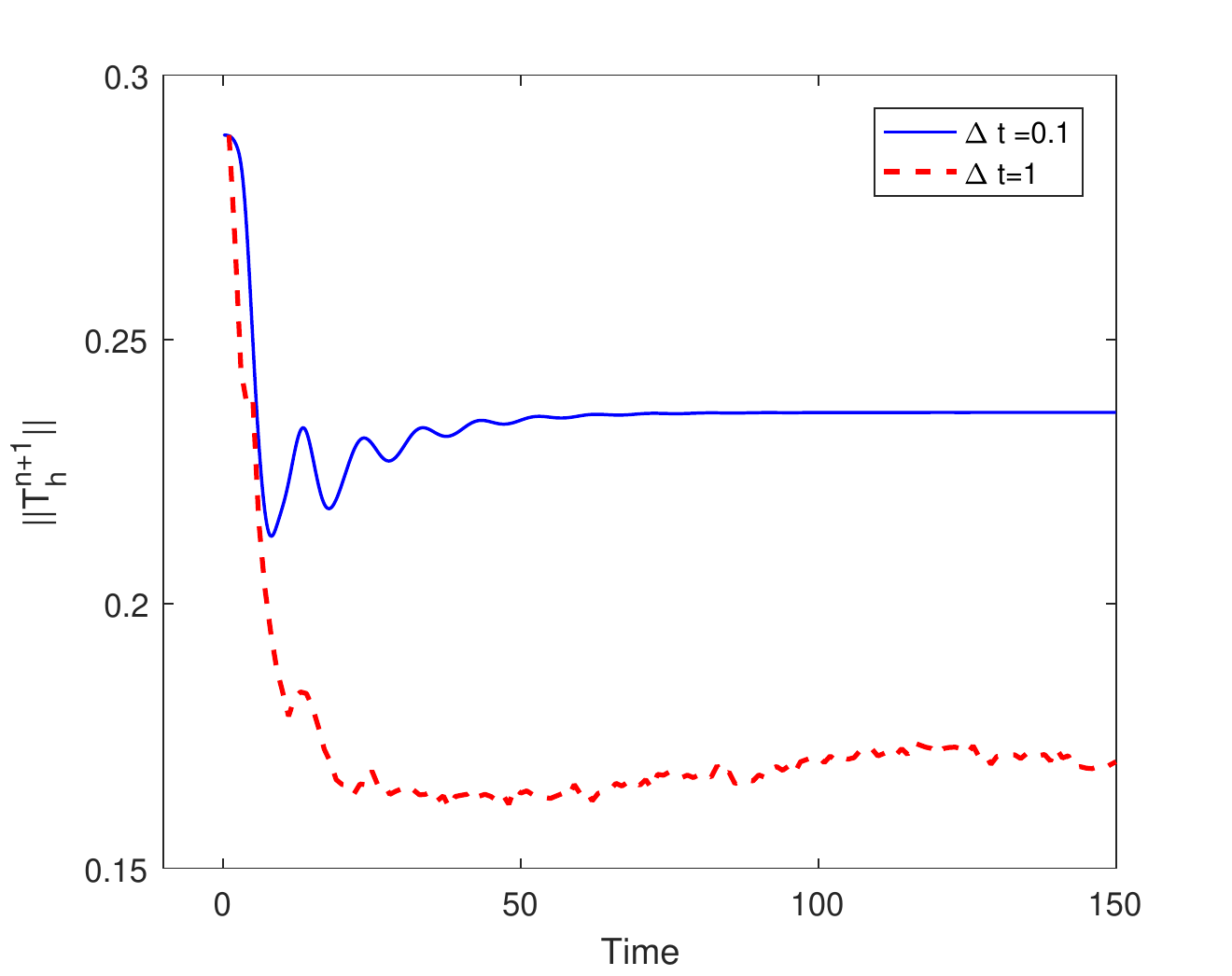}	
	}}
	\caption{\label{fig:nc} Evolution of the $L^{2}$ norm of the solution for varying viscosities. Velocity for $\nu = 0.01,\,\,\nu = 0.001$ (upper left to right) and  temperature for $\nu = 0.01,\,\,\nu = 0.001$ (lower left to right). }
\end{figure}


A clear observation from  Figure \ref{fig:nc}, both velocity and temperature solutions are globally in time bounded for varying viscosity instances. As we expect, due to the nature of these flows, narrow oscillation intervals could be observed for smaller $\nu$ values but still we can conclude that the solutions are all long time stable as theory predicts.

Similar to the Navier-Stokes case, we present the comparison table of CPU times in order to feel the computational advantage of the BLEBDF in Table \ref{table:cpunc}. Since the degree of freedom of overall system has been increased in natural convection case, we observe more visible CPU time differences in our comparison with the classical BDF2 scheme. Also we have used a rather short time interval for this test case. One could conclude from this comparison that, using the BLEBDF scheme instead of BDF2 scheme will yield a computational time advantage when smaller time step sizes are used on longer time intervals.

\begin{table}[hh!]
\centering
\begin{tabular}{|c c c|}
\hline
  $\Delta t$ & BDF2 & BLEBDF \\ \hline\hline
  1 & 6.71 & 6.27 \\
  0.1 & 61.8 & 59 \\
  0.01 & 628& 610\\
  \hline
\end{tabular}
\caption{ Comparison of the CPU-times (seconds) of classical BDF2 scheme and BLEBDF for natural convection with $\nu=0.001$ on a time interval $[0,10]$.}
\label{table:cpunc}
\end{table}

\section{Long time stability of double-diffusive convection with BLEBDF}

Under the assumption of Boussinesq approximation, we consider a fluid flow in $\Omega$  with a polygonal boundary $\partial\Omega= \Gamma_T\cup\Gamma_B $ with $\Gamma_T\cap\Gamma_B = \emptyset $. The governing equations of double-diffusive convection known as also Darcy-Brinkman system are given by (see e.g. \cite{goy}),
\begin{eqnarray}\label{dnse}
 \begin {array}{rcll}
\bu_t -\nu \Delta \bu+ (\bu\cdot\nabla)\bu +Da^{-1}\bu+ \nabla p &=& (\beta_T T + \beta_C C){\bf g} + \bff  &
 \mathrm{in }\ \Omega, \\
  \nabla \cdot \bu&=& 0& \mathrm{in }\ \Omega,\\
 T_t-\gamma \Delta T+\bu\cdot \nabla T&=& \kappa & \mathrm{in }\  \Omega,\\
 C_t-{D_c}\Delta C+\bu\cdot \nabla C&=&\zeta& \mathrm{in }\   \Omega,\\
\bu(0,\bx)= \bu_0,\ T(0,\bx)=  T_0 , \ C(0,\bx) &=& C_0 & \mathrm{in }\ \Omega, \\
\bu= \mathbf{0} \ \mathrm{on }\ \partial \Omega,\ T,C= 0\, \ \mathrm{on}\ \Gamma_T,  \, \dfrac{\partial T}{\partial {\bf n}}=0, \dfrac{\partial C}{\partial {\bf n}}&=&0& \mathrm{on} \ \Gamma_B.
 \end{array}
\end{eqnarray}
Besides the parameters defined earlier, $C$ is the concentration, $C_0$ the initial fluid concentration and $\zeta$  the body force for concentration equation. We also have the Darcy number $Da$, the thermal conductivity  $\kappa$, the mass diffusivity ${D_c} > 0$, ${\bf g}$ the gravitational acceleration vector, the thermal and solutal expansion coefficients are $\beta_T$, $\beta_C$, respectively. The dimensionless parameters are the buoyancy ratio $N=\dfrac{\beta_C \Delta C}{\beta_T \Delta T}$, the Schmidt number $Sc=\dfrac{\nu}{D_c}$, Prandtl number $Pr=\dfrac{\nu}{\gamma}$, the Darcy number $Da=\dfrac{K}{H^2}$, the Lewis number $Le=\dfrac{Sc}{Pr}$ and the thermal Rayleigh number $Ra=\dfrac{\bfg \beta_T \Delta T H^3}{\nu \gamma }$. Here the cavity height is $H$ and permeability is $K$. $\Delta T$ and $\Delta C$ are the temperature and concentration differences, respectively.

The fully discrete approximation of (\ref{dnse}) based on the BLEBDF scheme reads as; for each time level, we look for approximations $(\bu_h^{n+1},T_h^{n+1},C_h^{n+1})\in (\bfX_h, W_h,\Psi_h)$, satisfying
\begin{eqnarray}\label{ddbl}
\lefteqn{\left(\dfrac{\frac{5}{3}\bu_h^{n+1}-\frac{5}{2}\bu_h^n+\bu_h^{n-1}-\frac{1}{6}\bu_h^{n-2}}{\Delta t},\bv_h\right) + \nu(\nabla{\bu_h^{n+1}},\nabla{\bv_h}) + b_1(\bu^{*},\bu_h^{n+1},\bv_h)}\nonumber\\&&+Da^{-1}(\bu_h^{n+1},\bv_h)=\Big((\beta_{T}T^{*}+ \beta_{C}C^{*} ){\bf g}, \bv_h\Big) + (\bff^{n+1}, \bv_h) , \label{mnud}\\[15pt]
\lefteqn{\left(\dfrac{\frac{5}{3}T_h^{n+1}-\frac{5}{2}T_h^n+T_h^{n-1}-\frac{1}{6}T_h^{n-2}}{\Delta t},S_h\right)+b_2(\bu^{*},T_h^{n+1},S_h)+ \kappa(\nabla T_h^{n+1}, \nabla S_h)}\nonumber\\&=&(\gamma^{n+1}, S_h),\label{mntd}\\[15pt]
\lefteqn{\left(\dfrac{\frac{5}{3}C_h^{n+1}-\frac{5}{2}C_h^n+C_h^{n-1}-\frac{1}{6}C_h^{n-2}}{\Delta t},\phi_h\right)+b_3(\bu^{*},C_h^{n+1},\phi_h)+ D_c(\nabla C_h^{n+1}, \nabla \phi_h)}\nonumber\\
&=&(\zeta^{n+1}, \phi_h).\label{mncd}
\end{eqnarray}
for all $(\bv_h^{n+1},S_h^{n+1},\phi_h^{n+1})\in (\bfX_h, W_h,\Psi_h)$ where $\bu^{*}= \bu_h^{n+1} = 3\bu_h^n-3\bu_h^{n-1}+\bu_h^{n-2}$, $T^{*}= T_h^{n+1} = 3T_h^n-3T_h^{n-1}+T_h^{n-2}$ and $C^{*}= C_h^{n+1} = 3C_h^n-3C_h^{n-1}+C_h^{n-2} $

This section is devoted to prove the long time stability of the solutions of (\ref{mnud})-(\ref{mncd}).

\begin{theorem}(Unconditional long time stability of double-diffusive convection in $L^2$) \label{Lem:stamd} Assume $\bff \in {L^{\infty}(\mathbb{R}_+,\bfV_h^*)}$, $\gamma \in {L^{\infty}(\mathbb{R}_+,W_h^*)}$ and $\zeta \in {L^{\infty}(\mathbb{R}_+,\Psi_h^*)}$, then the approximation (\ref{mnud})-(\ref{mncd}) is long time stable in the following sense: for any $\Delta t>0$,

	\begin{eqnarray}
\lefteqn{	\|{ \bf \mathcal{U}}_{n+1}\|_G^2	+\|{ \bf \mathcal{T}}_{n+1}\|_G^2+\|{ \bf \mathcal{C}}_{n+1}\|_G^2+\dfrac{\nu\Delta t}{4}\|\nabla{\bu_h^{n+1}}\|^2+\dfrac{\kappa\Delta t}{4}\|\nabla{T_h^{n+1}}\|^2+\dfrac{D_c\Delta t}{4}\|\nabla{C_h^{n+1}}\|^2}	\nonumber\\
&&+\dfrac{Da^{-1}\Delta t}{4}\|{\bu_h^{n+1}}\|^2\nonumber\\
 &\leq&(1+\alpha)^{-(n+1)}\Big(\|{ \bf \mathcal{U}}_{0}\|_G^2+\frac{\nu\Delta t}{4}\|\nabla{\bu_h^{0}}\|^2+\frac{\nu\Delta t}{16}\|\nabla{\bu_h^{-1}}\|^2+\dfrac{Da^{-1}\Delta t}{4}\|{\bu_h^{0}}\|^2+\dfrac{Da^{-1}\Delta t}{16}\|{\bu_h^{-1}}\|^2\Big)\nonumber\\
	&&+\Big(K_{T,\alpha}+(1+\beta)^{-1}\Big)\Big((1+\beta)^{-n}(\|{ \bf \mathcal{T}}_{0}\|_G^2+\frac{\kappa\Delta t}{4}\|\nabla{T_h^{0}}\|^2+\frac{\kappa\Delta t}{16}\|\nabla{T_h^{-1}}\|^2)\Big)\nonumber\\
	&&\nonumber\\
	&&+\Big(K_{C,\alpha}+(1+\delta)^{-1}\Big)\Big((1+\delta)^{-n}(\|{ \bf \mathcal{C}}_{0}\|_G^2+\frac{D_c\Delta t}{4}\|\nabla{C_h^{0}}\|^2+\frac{D_c\Delta t}{16}\|\nabla{C_h^{-1}}\|^2)\Big)\nonumber	\\
		&&+(K_{T,\alpha}+1)\max \{ \dfrac{8C_p^2}{C_l\kappa^2}, \dfrac{2\kappa^{-1}\Delta t}{3}\} \|\gamma\|_{L^{\infty}(\mathbb{R}_+,W_h^*)}^2\nonumber\\
	&&+(K_{C,\alpha}+1)\max \{ \dfrac{8C_p^2}{C_lD_c^2}, \dfrac{2D_c^{-1}\Delta t}{3}\} \|\zeta\|_{L^{\infty}(\mathbb{R}_+,\Psi_h^*)}^2\nonumber\\
	&&+ \max \{ \dfrac{8C_p^2}{C_l\nu(\nu+C_p^2Da^{-1})}, \dfrac{2\nu^{-1}\Delta t}{3}\}\|\bff\|_{L^{\infty}(\mathbb{R}_+,\bfV_h^*)}^2 \label{lemmud},
	\end{eqnarray}
	where
	$K_{T,\alpha}= \dfrac{49C_uDa\Delta t \|{\bf g}\|_{\infty}^2\beta_T^2 }{\alpha}$, $K_{C,\alpha}=\dfrac{49C_uDa\Delta t \|{\bf g}\|_{\infty}^2 \beta_C^2}{\alpha}$,
$\mathcal{U}_0=[\bu_h^0 \quad \bu_h^{-1} \quad \bu_h^{-2}]$,  \\ $	\mathcal{T}_0=[T_h^0 \quad T_h^{-1} \quad T_h^{-2}]$,
	$\mathcal{C}_0=[C_h^0 \quad C_h^{-1} \quad C_h^{-2}]$,
	$	\alpha=\min\{\dfrac{C_l(\nu+C_p^2Da^{-1})\Delta t}{16C_p^2},\dfrac{3}{4}\},$
 $\beta=\min\{\dfrac{C_l\kappa \Delta t}{16C_p^2}, \dfrac{3}{4}\}$, $ \delta=\min\{\dfrac{C_lD_c\Delta t}{16C_p^2}, \dfrac{3}{4}\}$.
\end{theorem}

\begin{proof}
Arguing in the same way as for (\ref{bl5tn}), letting $S_h=T_h^{n+1}$ in (\ref{mntd}) and $\phi_h=C_h^{n+1}$ in (\ref{mncd}) gives  	 
\begin{eqnarray}
\lefteqn{	\|{ \bf \mathcal{T}}_{n+1}\|_G^2	 +\dfrac{\kappa\Delta t}{4}\|\nabla{T_h^{n+1}}\|^2+\dfrac{\kappa\Delta t}{16}\|\nabla{T_h^{n}}\|^2}\nonumber
	\\
&&	\leq(1+\beta)^{-(n+1)}(\|{ \bf \mathcal{T}}_{0}\|_G^2+\dfrac{\kappa\Delta t}{4}\|\nabla{T_h^{0}}\|^2
	+\dfrac{\kappa\Delta t}{16}\|\nabla{T_h^{-1}}\|^2)+\dfrac{\kappa^{-1}\Delta t}{2\beta} \|\gamma\|_{L^{\infty}(\mathbb{R}_+,W_h^*)}^2.\label{bl5td}
	\end{eqnarray}
for the temperature	where $\beta=min\{\dfrac{C_l\kappa \Delta t}{16C_p^2}, \dfrac{3}{4}\}$ and
	\begin{eqnarray}
\lefteqn{	\|{ \bf \mathcal{C}}_{n+1}\|_G^2	 +\dfrac{D_c\Delta t}{4}\|\nabla{C_h^{n+1}}\|^2+\dfrac{D_c\Delta t}{16}\|\nabla{C_h^{n}}\|^2}\nonumber
	\\
	&\leq&(1+\delta)^{-(n+1)}(\|{ \bf \mathcal{C}}_{0}\|_G^2	+\dfrac{D_c\Delta t}{4}\|\nabla{C_h^{0}}\|^2
	+\dfrac{D_c\Delta t}{16}\|\nabla{C_h^{-1}}\|^2)+\dfrac{D_c^{-1}\Delta t}{2\delta} \|\zeta\|_{L^{\infty}(\mathbb{R}_+,\Psi_h^*)}^2\label{bl5c}.
	\end{eqnarray}
for the concentration	where $\delta=\min\{\dfrac{C_lD_c \Delta t}{16C_p^2},\dfrac{3}{4}\}$ and  ${ \bf \mathcal{C}}_{n+1}=[C_h^{n+1} \quad C_h^{n} \quad C_h^{n-1}]^\top$ and ${ \bf \mathcal{C}}_{n}=[C_h^{n} \quad C_h^{n-1} \quad C_h^{n-2}]^\top$. Next, a bound for the velocity begins with letting $\bv_h=\bu_h^{n+1}$ in (\ref{mnud}), utilizing the similar ideas, one finds
	\begin{gather}
\|{ \bf \mathcal{U}}_{n+1}\|_G^2-\|{ \bf \mathcal{U}}_{n}\|_G^2+\dfrac{1}{12}\|\bu_h^{n+1}-3\bu_h^{n}+3\bu_h^{n-1}-\bu_h^{n-2}\|^2	
+\dfrac{\nu\Delta t}{2}\|\nabla{\bu_h^{n+1}}\|^2+\frac{Da^{-1}\Delta t}{2}\|\bu_h^{n+1}\|^2\nonumber\\
	\leq Da\Delta t\|{\bf g}\|_{\infty}^2(\beta_{T}^2\|T^{*}\|^2+ \beta_{C}^2\|C^{*}\|^2) + \frac{\nu^{-1} \Delta t}{2}\|\bff^{n+1}\|_{\bfV_h^*}^2
	\label{blu1d}
	\end{gather}
As for (\ref{blu2}), adding both of side $\dfrac{\nu\Delta t}{4}\|\nabla{\bu_h^{n}}\|^2+\dfrac{\nu\Delta t}{16}\|\nabla{\bu_h^{n-1}}\|^2+\dfrac{Da^{-1}\Delta t}{4}\|{\bu_h^{n}}\|^2+\dfrac{Da^{-1}\Delta t}{16}\|{\bu_h^{n-1}}\|^2$ and arguing exactly in the same way, the  (\ref{blu1d}) yields
	\begin{eqnarray}
\lefteqn{	(1+\alpha)(\|{ \bf \mathcal{U}}_{n+1}\|_G^2	 +\dfrac{\nu\Delta t}{4}\|\nabla{\bu_h^{n+1}}\|^2+\dfrac{\nu\Delta t}{16}\|\nabla{\bu_h^{n}}\|^2+\dfrac{Da^{-1}\Delta t}{4}\|{\bu_h^{n+1}}\|^2+\dfrac{Da^{-1}\Delta t}{16}\|{\bu_h^{n}}\|^2)}\nonumber
\\
&\leq&(\|{ \bf \mathcal{U}}_{n}\|_G^2+\dfrac{\nu\Delta t}{4}\|\nabla{\bu_h^{n}}\|^2
	+\dfrac{\nu\Delta t}{16}\|\nabla{\bu_h^{n-1}}\|^2+\dfrac{Da^{-1}\Delta t}{4}\|{\bu_h^{n}}\|^2+\dfrac{Da^{-1}\Delta t}{16}\|{\bu_h^{n-1}}\|^2) \nonumber
	\\
	&&+Da\Delta t \|{\bf g}\|_{\infty}^2(\beta_{T}^2\|T^{*}\|^2+ \beta_{C}^2\|C^{*}\|^2) + \frac{\nu^{-1} \Delta t}{2}\|\bff^{n+1}\|_{\bfV_h^*}^2 \label{blu2d}
	\end{eqnarray}
where $	\alpha=\min\{\dfrac{\Delta tC_l}{16}(\dfrac{\nu}{C_p^2}+Da^{-1}),\dfrac{3}{4}\}$. Utilizing (\ref{eqv}), and the definitions of $\mathcal{T}^*$,  $\mathcal{C^*}$, we get $\|\mathcal{T}^*\|\leq 7\|\mathcal{T}_n\|\leq 7\sqrt{C_u}\|\mathcal{T}_n\|_G$ and $\|\mathcal{C}^*\|\leq 7\|\mathcal{C}_n\|\leq 7\sqrt{C_u}\|\mathcal{C}_n\|_G$.
Putting $n$ instead of $(n+1)$ in (\ref{bl5td}) and (\ref{bl5c}) and inserting them in (\ref{blu2d}) yields
	\begin{eqnarray}
\lefteqn{	(1+\alpha)(\|{ \bf \mathcal{U}}_{n+1}\|_G^2	 +\dfrac{\nu\Delta t}{4}\|\nabla{\bu_h^{n+1}}\|^2+\dfrac{\nu\Delta t}{16}\|\nabla{\bu_h^{n}}\|^2+\dfrac{Da^{-1}\Delta t}{4}\|{\bu_h^{n+1}}\|^2+\dfrac{Da^{-1}\Delta t}{16}\|{\bu_h^{n}}\|^2)}\nonumber
\\
&\leq&(\|{ \bf \mathcal{U}}_{n}\|_G^2+\dfrac{\nu\Delta t}{4}\|\nabla{\bu_h^{n}}\|^2
+\dfrac{\nu\Delta t}{16}\|\nabla{\bu_h^{n-1}}\|^2+\dfrac{Da^{-1}\Delta t}{4}\|{\bu_h^{n}}\|^2+\dfrac{Da^{-1}\Delta t}{16}\|{\bu_h^{n-1}}\|^2) \nonumber
\\
&&+49C_uDa\Delta t \|{\bf g}\|_{\infty}^2 \bigg( \beta_T^2\Big( (1+\beta)^{-n}(\|{ \bf \mathcal{T}}_{0}\|_G^2+\dfrac{\kappa\Delta t}{4}\|\nabla{T_h^{0}}\|^2
+\dfrac{\kappa\Delta t}{16}\|\nabla{T_h^{-1}}\|^2)
+\dfrac{\kappa^{-1}\Delta t}{2\beta} \|\gamma\|_{L^{\infty}(\mathbb{R}_+,W_h^*)}^2\Big)\nonumber
\\
&&+\beta_C^2\Big((1+\delta)^{-n}(\|{ \bf \mathcal{C}}_{0}\|_G^2
	+\dfrac{D_c\Delta t}{4}\|\nabla{C_h^{0}}\|^2+\dfrac{D_c\Delta t}{16}\|\nabla{C_h^{-1}}\|^2)+\dfrac{D_c^{-1}\Delta t}{2\delta} \|\zeta\|_{L^{\infty}(\mathbb{R}_+,\Psi_h^*)}^2\Big)\bigg) \nonumber
	\\
	&&+ \frac{\nu^{-1} \Delta t}{2}\|\bff^{n+1}\|_{\bfV_h^*}^2 \label{blu3d}
	\end{eqnarray}
	Applying induction 
	and adding (\ref{blu3d}) with (\ref{bl5td}) and (\ref{bl5c}) gives stated result (\ref{lemmud}).
	
\end{proof}

\subsection{Numerical experiment for double-diffusive convection}

Among all the multiphysics flow examples issued in this study, the most challenging kind of flow is the double-diffusive convection case due to the highly oscillatory nature of its solutions \cite{goy,qin,R15}. As the Rayleigh number increases,  complex behavior of the solutions becomes more visible. However, we could still show the long time stability of the solutions of this system with the moderate Rayleigh number of $1000$ in a cavity convection case with the problem parameters, $N=0.8,\, Le=2,\, Pr=1$. We study in a rectangular domain $\Omega=(0,1)\times (0,2)$ with a coarse mesh resolution of $10 \times 20$. We note that, in the test, the Darcy number $Da$ is taken to be infinity for convenience.

Similar to the the previous test, vertical boundaries are kept at different temperature and concentration values imposed as Dirichlet boundary conditions whereas, the velocity boundary conditions are still no-slip all over the boundary. For a better understanding , we illustrate the details of the computational domain in Figure \ref{fig:dbdomain} as in previous case.

\begin{figure}[htb]
\centerline{\hbox{
	\includegraphics[width=0.45\linewidth]{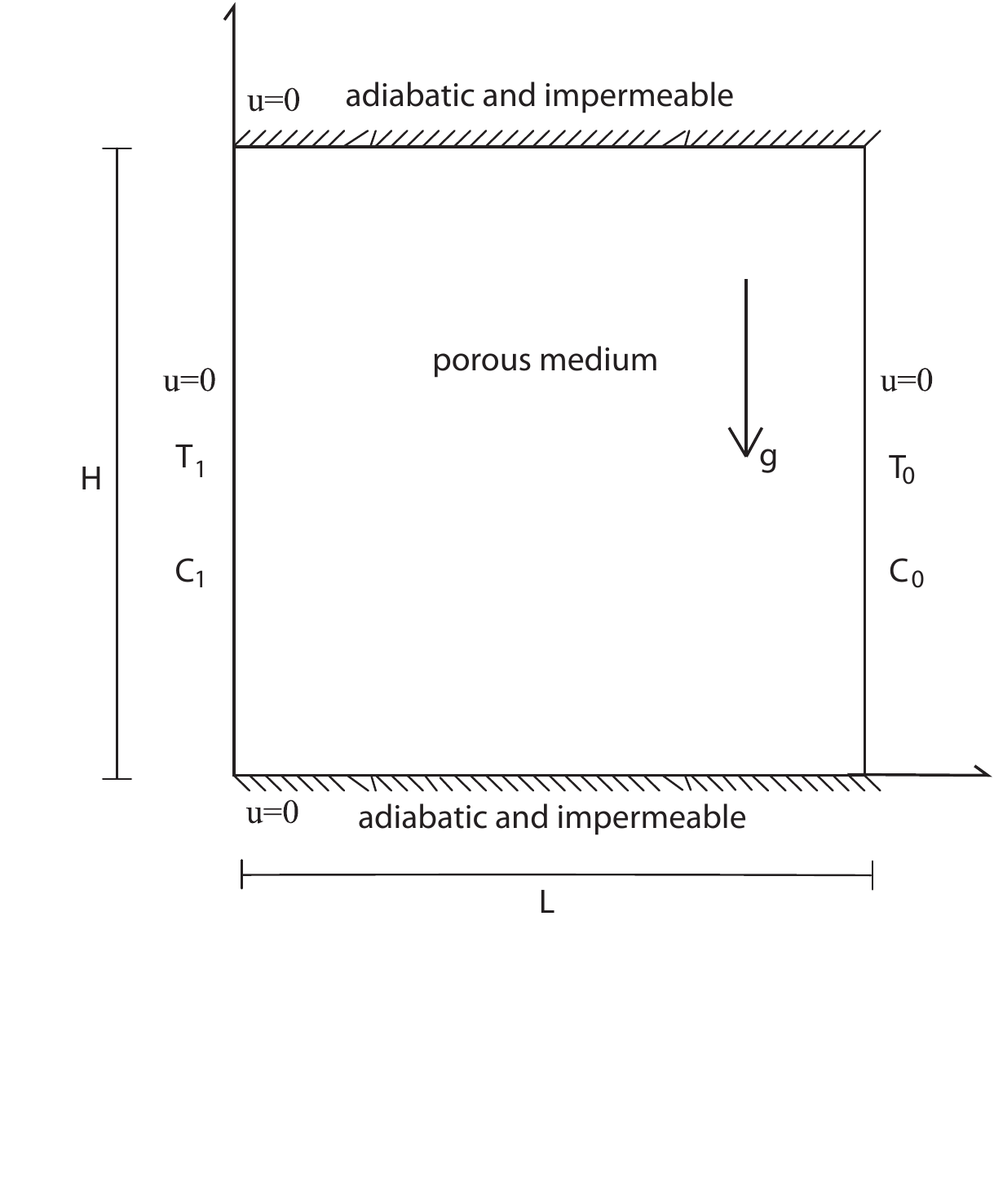}
}}
\caption{\label{fig:dbdomain} The computational domain for the double-diffusive convection test example.}
\end{figure}
\begin{figure}[h!]
	\centerline{\hbox{
			\includegraphics[width=0.35\linewidth]{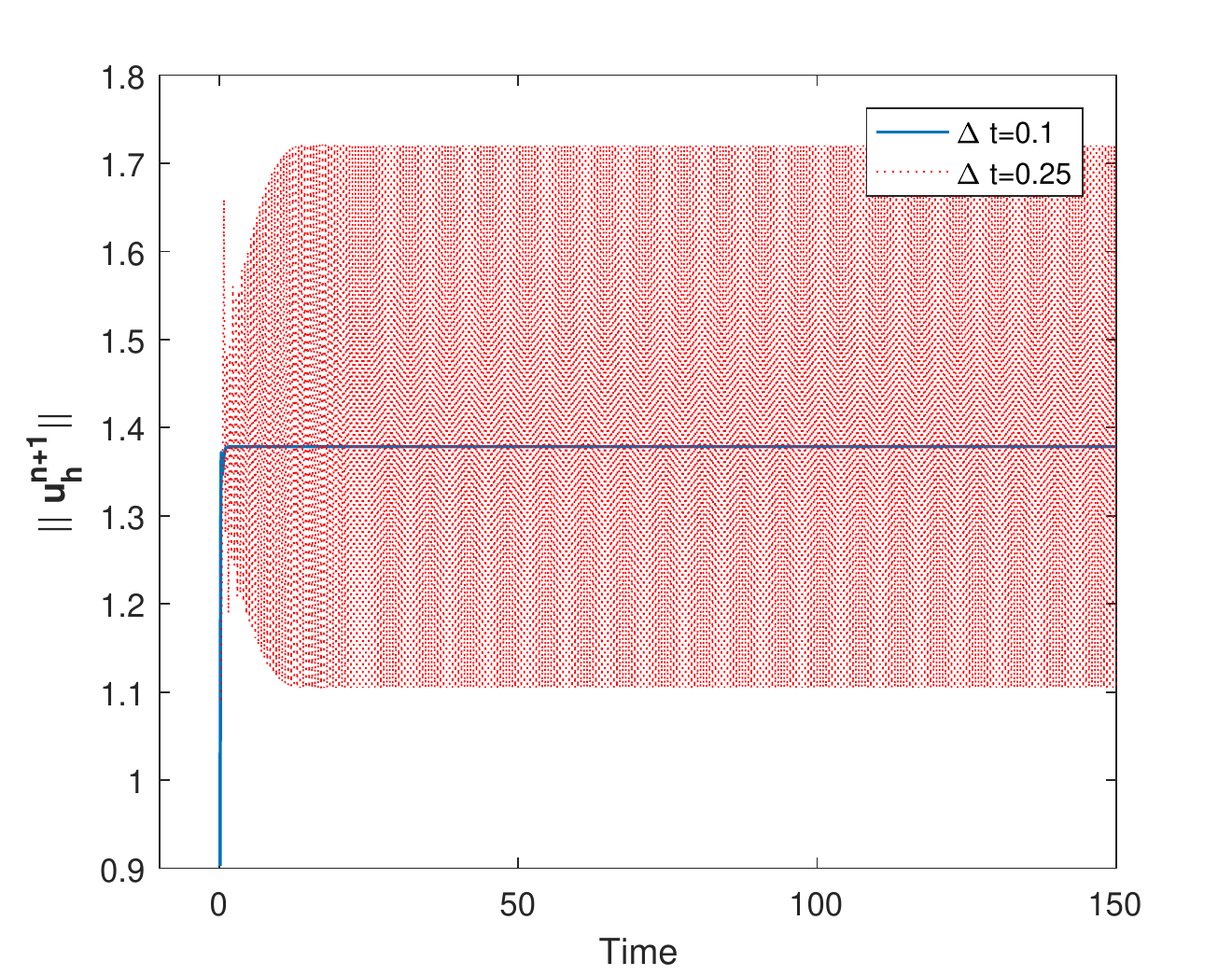}
			\includegraphics[width=0.35\linewidth]{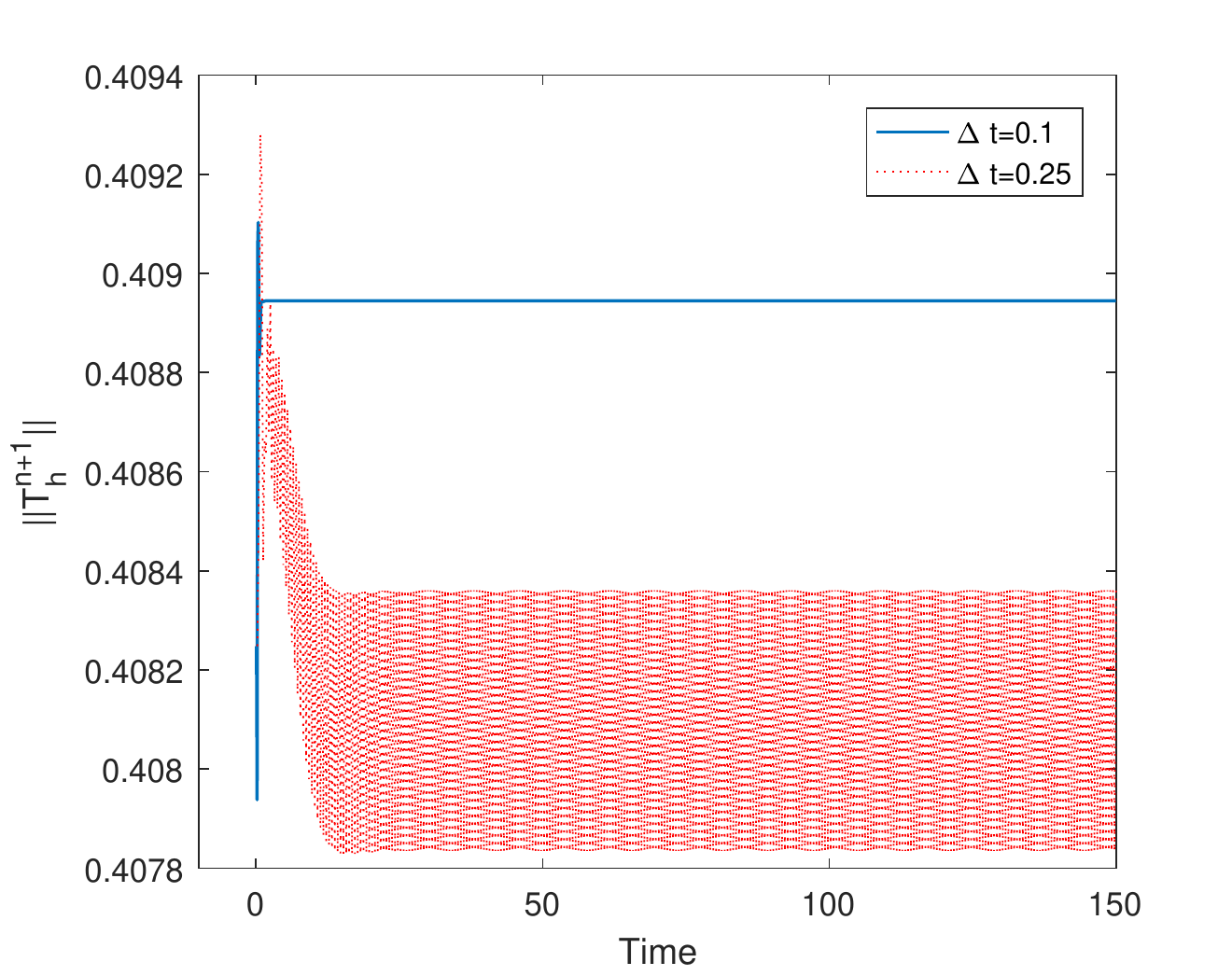}		
          	\includegraphics[width=0.35\linewidth]{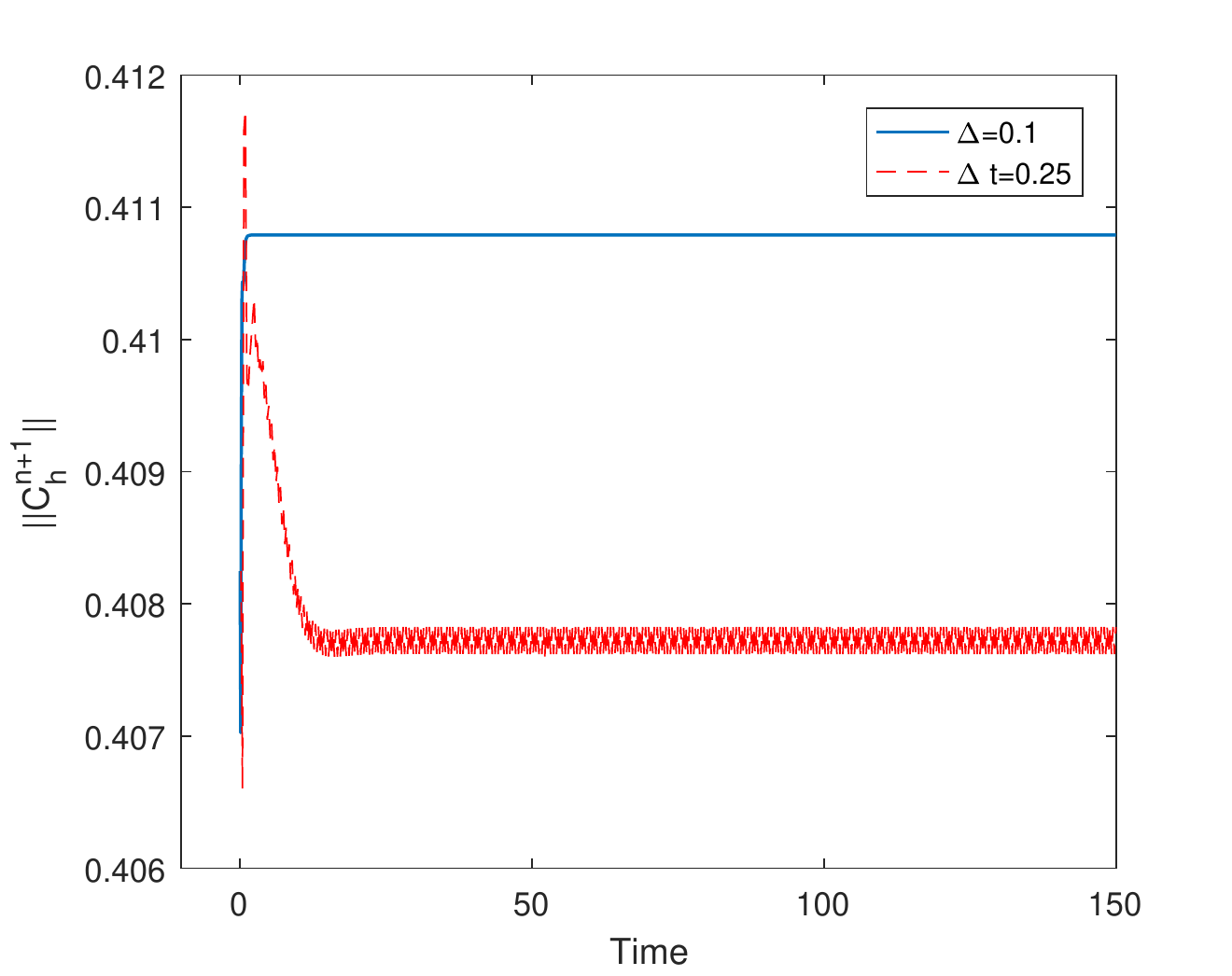}
          }}
	\caption{\label{fig:db} Evolution of the $L^{2}$ norm of the solution for $Ra=10^3,\, Le=2,\, Pr=1,\, N=0.8$. }
\end{figure}
In Figure \ref{fig:db}, the evolutions of the norms of each variable are illustrated. It can be observed that the scheme results to for larger $\Delta t$ values, solutions are stable inside an interval due to the complex attitude of the solution as explained earlier. However, the unconditional stability property of the solutions have been presented through this example.

In Table \ref{table:cpudb}, we present the CPU time comparison table for the test problem. The most visible CPU time differences are occurred in this case due to the increment in the degrees of freedom. As one more equation has been coupled to the system compared with the previous test case, we feel the superiority of the BLEBDF scheme when compared to the classical BDF2 scheme especially for smaller $\Delta t$. So the BLEBDF is clearly preferable when solving coupling systems with small time step sizes on longer time intervals.
\begin{table}[hh!]
\centering
\begin{tabular}{|c c c|}
\hline
  $\Delta t$ & BDF2 & BLEBDF\\ \hline\hline
  1 & 9.41 & 9 \\
  0.1 & 92 & 88 \\
  0.01 & 1075& 860\\
  \hline
\end{tabular}
\caption{ Comparison of the CPU-times (seconds) of classical BDF2 scheme and BLEBDF scheme for the double-diffusive convection test problem with $Ra=10^3$ on a time interval $[0,10]$.}
\label{table:cpudb}
\end{table}

\section{Conclusion}
This study deals with the long time stability of multiphysics flow problems including the Navier-Stokes equations, natural convection and double-diffusive convection systems with BLEBDF temporal discretization along with the finite element method in spatial discretization. Unconditional stability of the schemes has been proven and theoretical results are supported with various numerical examples. Computed CPU times suggest that this method noticeably saves time compared with the
classical BDF2 scheme for small time step sizes. These numerical experiments establish the strong evidence of the long time stability results of multiphysics flows presented in this study.

\bibliographystyle{amsplain}
\bibliography{reference}
\end{document}